\documentclass[11pt]{amsart}
\pdfoutput=1
\usepackage{array}
\usepackage{color}
\usepackage{enumerate}
\usepackage{listings}
\usepackage[colorlinks,linkcolor=black]{hyperref}

\usepackage{amscd}
\usepackage{amsthm}
\usepackage{amsmath}
\usepackage{amsfonts}
\usepackage{amssymb}
\usepackage{mathrsfs}
\usepackage{bm}

\usepackage{euscript}
\usepackage{epsfig}
\usepackage{graphicx}
\usepackage{latexsym}
\usepackage{pdfsync}
\usepackage{wasysym}
\usepackage{stmaryrd}
\usepackage{tikz-cd}
\usepackage{pgfplots}
\usepackage{mathrsfs}
\usetikzlibrary{arrows}

\newcommand{\rnum}[1]{\uppercase\expandafter{\romannumeral #1\relax}}

\newcommand{\exs}{\,\exists\,}
\newcommand{\foa}{\,\forall\,}

\newcommand{\emp}{\varnothing}

\newcommand{\R}{\mathbb{R}}
\newcommand{\Z}{\mathbb{Z}}

\newcommand{\N}{\mathbb{N}}

\newcommand{\const}{const}
\newcommand{\sgn}{\operatorname{sgn}}

\newcommand{\Lip}{\operatorname{Lip}}

\newcommand{\Abs}{\operatorname{Abs}}

\newcommand{\FF}{\mathcal{F}}

\newcommand{\HH}{\mathcal{H}}

\newcommand{\wh}{\widehat}
\newcommand{\wt}{\widetilde}
\newcommand{\cl}{\overline}

\theoremstyle{theorem}
\newtheorem{thm}{Theorem}[section]
\newtheorem{prop}[thm]{Proposition}

\theoremstyle{lemma}
\newtheorem{lemma}[thm]{Lemma}

\theoremstyle{corollary}
\newtheorem{cor}[thm]{Corollary}
\newtheorem{conj}[thm]{Conjecture}

\theoremstyle{definition}
\newtheorem{defi}[thm]{Definition}

\theoremstyle{example}
\newtheorem{example}[thm]{Example}

\theoremstyle{remark}
\newtheorem{rem}[thm]{Remark}

\theoremstyle{proof}

\author{Bobo Hua}
\address{Bobo Hua: School of Mathematical Sciences, LMNS,
	Fudan University, Shanghai 200433, China; Shanghai Center for
	Mathematical Sciences, Fudan University, Shanghai 200438,
	China.}
\email{bobohua@fudan.edu.cn}

\author{Florentin M\"unch}
\address{Florentin M\"unch: Max Planck Institute for Mathematics in the Sciences, Leipzig 04103, Germany}
\email{florentin.muench@mis.mpg.de}

\author{Haohang Zhang}
\address{Haohang Zhang: YMSC, Tsinghua University, Beijing 100084, China}
\email{zhanghh22@mails.tsinghua.edu.cn}
\title[Some variants of discrete positive mass theorems on graphs]{Some variants of discrete positive mass theorems on graphs}
\date{2024.1.30}

\begin{document}

\maketitle

\begin{abstract}
	Inspired by asymptotically flat manifolds, we introduce the concept of asymptotically flat graphs and define the discrete ADM mass on them. We formulate the discrete positive mass conjecture based on the scalar curvature in the sense of Ollivier curvature, and prove the positive mass theorem for asymptotically flat graphs that are combinatorially isomorphic to grid graphs. As a corollary, the discrete torus does not admit positive scalar curvature. We prove a weaker version of the positive mass conjecture: an asymptotically flat graph with non-negative Ricci curvature is isomorphic to the standard grid graph. Hence the combinatorial structure of an asymptotically flat graph is determined by the curvature condition, which is a discrete analog of the rigidity part for the positive mass theorem. The key tool for the proof is the discrete harmonic function of linear growth associated with the salami structure introduced in \cite{salami,salami2}.

\end{abstract}

%\tableofcontents	

\section{Introduction}\label{sec:intro}
	\subsection{The positive mass theorem on manifolds}
	The positive mass theorem is a significant result about the gravitational energy of an isolated system in both general relativity and Riemannian geometry \cite{ADM1, ADM2, ADM3}. In 1979, Schoen-Yau first proved the three-dimensional case using the minimal surface method \cite{SY1, SY3}. Witten provided a new proof for spin manifolds of arbitrary dimensions \cite{Witten}. The general case has been proven up to seven dimensions \cite{S2}; see \cite{Lohk2006, SY17} for related results.
	
	The positive mass theorem plays a crucial role in solving classical problems in mathematics. For instance, Schoen used the positive mass theorem to solve the Yamabe problem, which is an important breakthrough in the field of nonlinear partial differential equations \cite{S1}. Bray used the positive mass theorem to prove the Penrose inequality, which provides a stronger conclusion than the positive mass theorem \cite{B1, BL}.
	
	Nowadays, the positive mass theorem is still an active topic in geometry. Huang-Wu provided a proof for hypersurfaces under weaker asymptotic conditions in 2011 \cite{HW}. The case of K\"ahler manifolds was proven by Hein-LeBrun \cite{Kahler}. Recently, there are new proofs for the three-dimensional case using Ricci flows \cite{Ricci} and level sets of harmonic functions \cite{B2}.
	
	The positive mass theorem is formulated in Riemannian geometry as follows. First, we need to impose an asymptotic condition on the manifold.
	\begin{defi}\cite{S2}
		Let $(M^n,g)$ be a Riemannian manifold, $n\ge 3$. $(M,g)$ is said to be \textbf{asymptotically flat} if it satisfies the following two conditions.
		\begin{enumerate}[(i)]
			\item There exists a compact subset $K\subset M$ and a diffeomorphism $\Phi: M\setminus K\approx \R^n\setminus \cl{B}_1.$
			
			\item In the coordinate chart defined by $\Phi,$ the metric satisfies, as $|x|\to \infty$ with the Euclidean norm $|\cdot|$,
			\begin{align}
				g_{ij}(x)&=\delta_{ij}+O(|x|^{-p}),\qquad\qquad\qquad\qquad\nonumber\\
				|x||g_{ij,k}(x)|+|x|^2|g_{ij,kl}(x)|&=O(|x|^{-p}),\label{condition}\\
				|R(x)|&=O(|x|^{-q}),\nonumber
			\end{align}
			where $R$ is the scalar curvature, $p>\frac{n-2}{2},$ $q>n$ and 
			$$g_{ij,k}:=\partial_{k}g_{ij},\quad g_{ij,kl}:=\partial_{kl}g_{ij},\quad i,j,k,l=1,2,\cdots,n.$$
			
		\end{enumerate} 
		
	\end{defi}
	
	Then we can define the ADM mass for an asymptotically flat manifold $(M^n,g).$
	\begin{defi}\cite{S2}
		The \textbf{ADM mass} or the energy of an asymptotically flat manifold $(M^n,g)$ is defined as
		\begin{align}
			E(g)=\frac{1}{4(n-1)\omega_{n-1}}\lim_{\sigma\rightarrow \infty}\int_{S_\sigma}\sum_{i,j}(g_{ij,i}-g_{ii,j})\nu_jd\xi.\label{ADM}
		\end{align}
		where $S_{\sigma}:=\{|x|=\sigma\},$ $\omega_{n-1}$ is the surface area of the unit sphere $S^{n-1}$, $\nu=\sigma^{-1}x$ is the outer normal vector, and $d\xi$ is the area element of $S_{\sigma}$.
		
	\end{defi}
	
	The ADM mass is well-defined and finite. To verify this, one can derive the following formula by calculating scalar curvatures directly using connection coefficients.
	\begin{align}
		\int_{A_{\sigma,\tau}} R(g)dx=\int_{S_{\tau}-S_\sigma}\sum_{i,j}(g_{ij,i}-g_{ii,j})\nu_jd\xi+\int_{A_{\sigma,\tau}}O(|x|^{-2p-2})dx,\label{connection}
	\end{align}
	
	where $\sigma<\tau$, $A_{\sigma,\tau}:=\{\sigma\le |x|\le \tau\}$ and $\sigma$ is sufficiently large. The positive mass theorem \cite{S2} is stated as follows.
	\begin{thm}[The positive mass theorem]
		Suppose $(M^n,g)$ is an asymptotically flat manifold with $R(g)\ge 0$ and $n\le 7$. Then $E(g)\ge 0$ and $E(g)=0$ if and only if $(M,g)$ is isometric to $(\R^n,\delta),$ where $\delta$ is the Euclidean metric.
		
	\end{thm}

	A fundamentally important result is the rigidity part in the theorem, which concludes the stability of the Euclidean metric. In particular, this implies that the topology of the compact set $K$ has to be trivial by the zero-mass condition.

\subsection{The setting of weighted graphs}
We introduce the setting of weighted graphs. Let $(V,E)$ be a locally finite, simple, undirected graph. Two vertices $x,y$ are called neighbours, denoted by $x\sim y$, if there is an edge connecting $x$ and $y,$ i.e. $\{x,y\}\in E.$ A graph is called connected if for any $x,y\in V$ there is a path $\{z_i\}_{i=0}^n\subset V$ connecting $x$ and $y,$ i.e.
$$x = z_0 \sim \cdots \sim z_n = y.$$
Throughout this paper, we consider only connected graphs.  We define the combinatorial graph distance $d(x,y)$ to be 
$$d(x, y) :=
\inf\{n \mid x = z_0 \sim \cdots \sim z_n = y\}.$$
For any $R\in \Z_+,$ we denote by
$$B_R(x):=\{y\in V: d(y,x)\leq R\}$$
the ball of radius $R$ centered at $x.$  Two graphs $(V_i,E_i), i=1,2,$ are called combinatorially isomorphic if there is a bijection $\Phi:V_1\to V_2$ such that $\Phi(x)\sim \Phi(y)$ if and only if $x\sim y.$

For any $K_1,K_2\subset V,$ we define the set of edges between $K_1,K_2$ to be 
$$E(K_1,K_2):=\{\{x,y\}\in E: x\in K_1, y\in K_2\}.$$
For any subset $\Omega\subset V,$ we denote by $\partial \Omega:=E(\Omega,V\setminus \Omega)$ the edge boundary of $\Omega,$ and by
$$\delta\Omega:=\{y\in V\setminus\Omega: \exs x\in \Omega\ \mathrm{s.t.}\ y\sim x\}$$  the vertex boundary of $\Omega.$ We denote by $\overline{\Omega}:=\Omega\cup \delta\Omega$ the closure of $\Omega$.

Let $$w: E\to (0,\infty),\quad \{x,y\}\mapsto w(x,y)=w(y,x),$$ be a function that assigns positive weights to the edges of the graph. We call the triple $G=(V,E,w)$ a \emph{weighted graph}. Note that in general setting, weights are also assigned on vertices \cite{Kellerbook}, and we restrict to trivial vertex weights in this paper. Two graphs $G_i=(V_i,E_i,w_i)$, $i=1,2,$ are called \textit{weighted isomorphic} if there exists a combinatorially isomorphism $\Phi:V_1\rightarrow V_2$ preserving weights, i.e. $$w_2(\Phi(x),\Phi(y))=w_1(x,y),\quad \forall x\sim y.$$ 

For a weighted graph $G$ and any function $f:V\to \R,$ the Laplacian $\Delta$ is defined as
$$\Delta f (x):= \sum_{y\in V:y\sim x}w(x,y)\left(f(y)-f(x)\right), \quad\forall x\in V.$$
For $\Omega\subset V,$ a function $f$ on $V$ is called harmonic on $\Omega$ if $\Delta f=0$ on $\Omega.$  

%Edge $e=(x,y)\in E$ is an unordered pair. We write $x\sim y$ if $(x,y)\in E.$ We always consider simple graph with locally finiteness, i.e. there's no loops or multi-edges on graph, and $\deg x<\infty,$ $\foa x\in V.$ Sometimes we will abuse the notation to denote the vertex set $V$ by $G$ directly.
		
%		For weighted graph $G=(V,E,w,m),$ there's symmetric edge weight $w:V\times V\rightarrow [0,\infty)$ and vertex measure $m:V\rightarrow (0,\infty).$ The weights satisfy that $w(e)>0$ if and only if $e\in E.$ As for the measure, we assume $m\equiv 1$ for conciseness in this paper.
		
Let $\Z^n$ be a graph with the set of vertices $\Z^n:=\{x\in \R^n: x_i\in \Z\}$ and the set of edges $\{\{x,y\}: x,y\in \Z^n, \|x-y\|_1=1\},$ which is commonly known as the \textit{grid graph}. We call $(\Z^n, w)$ a weighted grid graph, or grid graph, with edge weights $w,$ and call it a standard grid graph when $w\equiv 1.$

For the grid graph $\Z^n,$ $x\in \Z^n$ and $r\in\Z_+,$ the cube (the sphere, resp.) of side-length $2r$  centered at $x$ is defined as
$$Q_r(x):=\{y\in\Z^n\,|\, \|y-x\|_\infty\le r\},$$
$$
(S_r(x):=\{y\in \Z^n\,|\,\|y-x\|_\infty=r\}, \mathrm{resp}.),
$$ 
where $\|x\|_{\infty}:=\max_{1\le i\le n}|x_i|$ is the $l^\infty$ norm of $x.$ Moreover, we define
 $$E_r(x):=\partial Q_r(x),\quad \wt{E}_r(x):=E(\delta Q_r(x), S_{r+2}),$$ which are illustrated in Figure~\ref{Fig:boundary1}.
For the case $x=0,$ we simply write $Q_r, E_r$.
				
		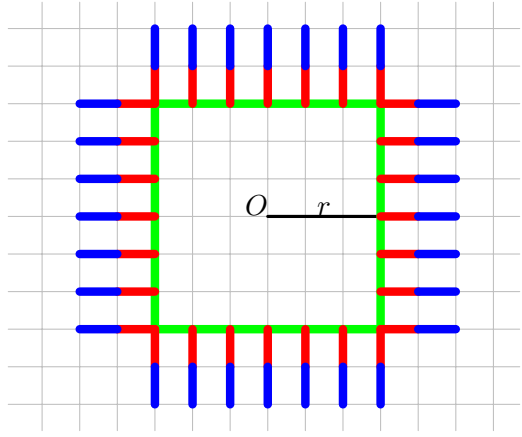
\begin{figure}[htbp]
			\centering
			% drawn by GeoGebra
			\definecolor{qqffqq}{rgb}{0,1,0}
			\definecolor{qqqqff}{rgb}{0,0,1}
			\definecolor{ffqqqq}{rgb}{1,0,0}
			\definecolor{ududff}{rgb}{0.3,0.3,1}
			\definecolor{cqcqcq}{rgb}{0.75,0.75,0.75}
			\begin{tikzpicture}[line cap=round,line join=round,>=triangle 45,x=1cm,y=1cm,scale=0.5]
				\draw [color=cqcqcq,, xstep=1cm,ystep=1cm] (-1.9,-0.7) grid (11.7,10.7);
%				\selectcolormodel{cmyk}
				\clip(-1.9,-0.7) rectangle (11.7,10.7);
				\draw [line width=1.2pt] (5,5)-- (8,5);
				\draw [line width=3.2pt,color=qqffqq] (2,8)-- (8,8);
				\draw [line width=3.2pt,color=qqffqq] (8,8)-- (8,2);
				\draw [line width=3.2pt,color=qqffqq] (8,2)-- (2,2);
				\draw [line width=3.2pt,color=qqffqq] (2,2)-- (2,8);
				\draw [line width=3.2pt,color=ffqqqq] (2,9)-- (2,8);
				\draw [line width=3.2pt,color=ffqqqq] (3,8)-- (3,9);
				\draw [line width=3.2pt,color=ffqqqq] (4,9)-- (4,8);
				\draw [line width=3.2pt,color=ffqqqq] (5,8)-- (5,9);
				\draw [line width=3.2pt,color=ffqqqq] (6,9)-- (6,8);
				\draw [line width=3.2pt,color=ffqqqq] (7,9)-- (7,8);
				\draw [line width=3.2pt,color=ffqqqq] (8,9)-- (8,8);
				\draw [line width=3.2pt,color=ffqqqq] (8,8)-- (9,8);
				\draw [line width=3.2pt,color=ffqqqq] (9,7)-- (8,7);
				\draw [line width=3.2pt,color=ffqqqq] (9,6)-- (8,6);
				\draw [line width=3.2pt,color=ffqqqq] (9,5)-- (8,5);
				\draw [line width=3.2pt,color=ffqqqq] (9,4)-- (8,4);
				\draw [line width=3.2pt,color=ffqqqq] (8,3)-- (9,3);
				\draw [line width=3.2pt,color=ffqqqq] (9,2)-- (8,2);
				\draw [line width=3.2pt,color=ffqqqq] (8,2)-- (8,1);
				\draw [line width=3.2pt,color=ffqqqq] (7,1)-- (7,2);
				\draw [line width=3.2pt,color=ffqqqq] (6,2)-- (6,1);
				\draw [line width=3.2pt,color=ffqqqq] (5,1)-- (5,2);
				\draw [line width=3.2pt,color=ffqqqq] (4,2)-- (4,1);
				\draw [line width=3.2pt,color=ffqqqq] (3,1)-- (3,2);
				\draw [line width=3.2pt,color=ffqqqq] (2,1)-- (2,2);
				\draw [line width=3.2pt,color=ffqqqq] (2,2)-- (1,2);
				\draw [line width=3.2pt,color=ffqqqq] (1,3)-- (2,3);
				\draw [line width=3.2pt,color=ffqqqq] (2,4)-- (1,4);
				\draw [line width=3.2pt,color=ffqqqq] (1,5)-- (2,5);
				\draw [line width=3.2pt,color=ffqqqq] (2,6)-- (1,6);
				\draw [line width=3.2pt,color=ffqqqq] (1,7)-- (2,7);
				\draw [line width=3.2pt,color=ffqqqq] (2,8)-- (1,8);
				\draw [line width=3.2pt,color=qqqqff] (2,9)-- (2,10);
				\draw [line width=3.2pt,color=qqqqff] (3,9)-- (3,10);
				\draw [line width=3.2pt,color=qqqqff] (4,10)-- (4,9);
				\draw [line width=3.2pt,color=qqqqff] (5,9)-- (5,10);
				\draw [line width=3.2pt,color=qqqqff] (6,10)-- (6,9);
				\draw [line width=3.2pt,color=qqqqff] (7,10)-- (7,9);
				\draw [line width=3.2pt,color=qqqqff] (8,10)-- (8,9);
				\draw [line width=3.2pt,color=qqqqff] (9,8)-- (10,8);
				\draw [line width=3.2pt,color=qqqqff] (10,7)-- (9,7);
				\draw [line width=3.2pt,color=qqqqff] (9,6)-- (10,6);
				\draw [line width=3.2pt,color=qqqqff] (10,5)-- (9,5);
				\draw [line width=3.2pt,color=qqqqff] (9,4)-- (10,4);
				\draw [line width=3.2pt,color=qqqqff] (10,3)-- (9,3);
				\draw [line width=3.2pt,color=qqqqff] (9,2)-- (10,2);
				\draw [line width=3.2pt,color=qqqqff] (8,1)-- (8,0);
				\draw [line width=3.2pt,color=qqqqff] (7,1)-- (7,0);
				\draw [line width=3.2pt,color=qqqqff] (6,1)-- (6,0);
				\draw [line width=3.2pt,color=qqqqff] (5,1)-- (5,0);
				\draw [line width=3.2pt,color=qqqqff] (4,1)-- (4,0);
				\draw [line width=3.2pt,color=qqqqff] (3,1)-- (3,0);
				\draw [line width=3.2pt,color=qqqqff] (2,1)-- (2,0);
				\draw [line width=3.2pt,color=qqqqff] (1,2)-- (0,2);
				\draw [line width=3.2pt,color=qqqqff] (0,3)-- (1,3);
				\draw [line width=3.2pt,color=qqqqff] (1,4)-- (0,4);
				\draw [line width=3.2pt,color=qqqqff] (0,5)-- (1,5);
				\draw [line width=3.2pt,color=qqqqff] (1,6)-- (0,6);
				\draw [line width=3.2pt,color=qqqqff] (0,7)-- (1,7);
				\draw [line width=3.2pt,color=qqqqff] (1,8)-- (0,8);
				\draw[color=black] (6.5,5.2) node {$r$};
				\draw[color=black] (4.7,5.3) node {$O$};
			\end{tikzpicture}
			\caption{Green edges indicate $S_r$ with $Q_r$ inside. Red edges indicate $E_r$ and blue edges indicate $\wt{E}_r.$}\label{Fig:boundary1}
		\end{figure}

Ollivier \cite{Ollivier} introduced a curvature notion on graphs via the optimal transport, later modified by Lin-Lu-Yau \cite{LLY,limit-free}. 
The Ollivier curvature has been extensively studied on graphs and there are many interesting applications; see e.g. \cite{Paeng12,CHO2013916,JostLiu14,BhattMuk15,Ni2015RicciCO,Sia2019OllivierRicciCM,Jost2019LiouvillePA,Eidi2019OllivierRC,Bai2020DiscoveringHS,Leal2020RicciCO,tee2021enhanced,Gosztolai2021UnfoldingTM,Mun22F}. Recall that the Ollivier curvature between two vertices $x,y$ (in particular for an edge $\{x,y\}\in E$) is given by, see \cite{limit-free},
		\begin{align}
			\kappa(x,y)=\inf\limits_{\begin{subarray}{c}
					f:B_1(x)\cup B_1(y)\rightarrow \R\\
					f\in \Lip(1)\\
					\nabla_{yx}f=1
			\end{subarray}}\nabla_{xy}\Delta f, \label{limit-freeR}
		\end{align}
		where 
		$$\nabla_{xy}f:=\frac{f(x)-f(y)}{d(x,y)},$$		
		$$\Lip(1):=\{f:V\rightarrow \R\,|\,|\nabla_{xy}f|\le 1,\foa x,y\in V\}.$$

In particular, it suffices to consider integer valued functions in (\ref{limit-freeR}) when the distance is taken to be the combinatorial distance \cite{limit-free}, i.e.
\begin{align}
	\kappa(x,y)=\inf\limits_{\begin{subarray}{c}
			f:B_1(x)\cup B_1(y)\rightarrow {\Z}\\
			f\in \Lip(1)\\
			\nabla_{yx}f=1
	\end{subarray}}\nabla_{xy}\Delta f,\label{limit-free}
\end{align}
		
Note that the Ollivier curvature is a discrete analog of the Ricci curvature in the continuous setting \cite{Ollivier}, and hence we call it the Ricci curvature on a graph.

We define the scalar curvature of the graph at a vertex $x$ as
$$
R(x)=\sum_{y\in V: y\sim x}\kappa(x,y).
$$
		
The advantage of the grid graph is that the Ricci curvature (and therefore its scalar curvature) has an explicit expression in term of edge weights; see Proposition~\ref{prop:ric}.	For a weighted grid graph $(\Z^n,w),$ let $e_i$ be the $i$-th unit coordinate vector, which is zero except the $i$-th element being one. Let $x_j^\pm:=x\pm e_j$ respectively. We define 

\begin{small}
\begin{align*}
	\Abs(x):=
	\sum_{\begin{subarray}{c}
			i,j=1\\
			i\neq j
	\end{subarray}}^n&\left(|w(x,x+e_i)-w(x_j^+,x^+_j+e_i)|+|w(x,x+e_i)-w(x_j^-,x_j^-+e_i)|\right.\\
	&+\left.|w(x,x-e_i)-w(x_j^+,x_j^+-e_i)|+|w(x,x-e_i)-w(x_j^-,x^-_j-e_i)|\right),
\end{align*}
\end{small}

which measures the distortion from the standard grid graph. The scalar curvature $R(x)$ can be expressed as linear terms in $w$ minus $\Abs(x)$, see Corollary \ref{cor:scal}.

\subsection{Discrete positive mass conjecture and theorem on graphs}

We introduce the definition of asymptotically flat graphs, which are discrete counterparts of asymptotically flat manifolds.
		\begin{defi}
			The graph $G=(V,E,w_G)$ is called \textbf{asymptotically flat} if it satisfies the following two conditions:
			\begin{enumerate}[(i)]
				\item There exists $p>n\geq 2,$ and a weighted grid graph $(\Z^n,{w})$ satisfying, as $|x|\to \infty,$
				\begin{align*}
					w(x,y)&=1+o(1),\quad \foa x\sim y,\\
					\Abs(x)&=O(|x|^{-p}),\quad \mathrm{and}\\
					|R(x)|&=O(|x|^{-p}).
			\end{align*}
		\item 
				There exists a finite subset $K\subset V,$ $r\in\Z_+$ such that the induced subgraph $(V\setminus K,w_G)$ is weighted isomorphic to $(\Z^n\setminus Q_r,{w}).$ Furthermore, denote the isomorphism by $\Phi,$ there is $E(K,\Phi^{-1}(\Z^n\setminus Q_{r+1}))=\emp.$
				
	\end{enumerate}
		\end{defi}

%		Let $d(x,y)$ be the combinatorial distance for $x,y\in V.$ We assume every graph is connected in this paper, i.e. $d(x,y)<\infty,$ $\foa x,y\in V.$ The Laplacian on graph is given by
	%	$$
	%	\Delta f(x)=\frac{1}{m(x)}\sum_y w(x,y)(f(y)-f(x)).
	%	$$
		
 Inspired by (\ref{ADM}) and (\ref{connection}), we define the discrete ADM mass of asymptotically flat graphs as follows.
		\begin{defi}[Discrete ADM mass]
			The ADM mass or the energy of an asymptotically flat graph $G$ is defined as
			\begin{align}
				E(G)=\frac{1}{2^nn}\lim_{r\rightarrow\infty} \left(\sum_{e\in E_r}w(e)-\sum_{\tau\in \wt{E}_r}w(\tau)\right),
			\end{align}
			
		\end{defi}		
		
We propose the following discrete positive mass conjecture. 
		
		\begin{conj}[Discrete positive mass conjecture]
			For an asymptotically flat graph $G$ with non-negative scalar curvature, we have $E(G)\ge 0.$ Furthermore, $E(G)=0$ if and only if $G$ is a standard grid graph.
			
		\end{conj}

%In this paper, we consider two cases, either the combinatorial structure is same as a grid graph, or the curvature condition is replaced by stronger one, that is the Ricci curvature.	

As we don't know the combinatorial structure in the finite subset $K,$ the scalar curvature is hard to compute in $K$. To approach the positive mass conjecture on graphs, we first consider the case that the combinatorial structure of $G$ is same as a grid graph.
In this case, we can prove the positive mass conjecture.

	\begin{thm}[Positive mass theorem on grid graphs]\label{thm:main1}
		For an asymptotically flat grid graph $G=(\Z^n,w)$ with $R\ge 0$, we have $E(G)\ge 0.$ Furthermore, $E(G)=0$ if and only if $G$ is a standard grid graph, i.e. $w\equiv 1.$
		
	\end{thm}

	The proof of the theorem crucially uses the grid structure, for which the scalar curvature has an explicit formula related to the ADM mass; see \eqref{key}. The positive mass theorem on grid graphs motivates the following result on tori.
	\begin{thm}\label{thm:tori}
	The total curvature of a discrete weighted torus is non-positive, provided that the torus satisfies the distance condition defined in Section~\ref{sec:tor}. Furthermore, if the torus has non-negative scalar curvature, then the scalar curvature has to be zero everywhere.
		
	\end{thm}

	Moreover, by imposing a stronger decay condition on an asymptotically flat grid graph, the mass has to be zero. The following corollary is a discrete analog of \cite[Theorem 3 (ii)]{mass} for manifolds, which indicates that the ADM mass of a graph is contained in the term $|x|^{-(n-2)}$.
	\begin{cor}\label{coro:strongdecay}
		Let $G$ be an asymptotically flat grid graph. Suppose $w(x,y)=1+O(|x|^{-p})$ and $|w(x,y)-w(y,z)|=O(|x|^{-p-1})$ for $x\sim y\sim z$ on the same line as $x\rightarrow\infty.$ If $p>n-2,$ then $E(G)=0$ and $w\equiv 1.$
	\end{cor}

Now we consider general asymptotically flat graphs without assuming the global grid structure. We prove a weaker version of the positive mass theorem by replacing the scalar curvature with the Ricci curvature.

	\begin{thm}\label{thm:mainric}
		Assume $G=(V,E,w)$ is an asymptotically flat graph with non-negative Ricci curvature, then $G$ has to be the standard grid graph.
		
	\end{thm}
	\begin{rem}
		\begin{enumerate}[(i)]
			\item For a Riemannian manifold with nonnegative Ricci curvature, a continuous analog of the above result follows from the Bishop-Gromov volume comparison and its rigidity. However, the polynomial volume growth and its rigidity are challenging open problems for the Ollivier curvature, which is of infinite-dimensional nature.
			\item This is a combinatorial rigidity result for asymptotically flat graphs. By the definition, the asymptotically flatness restricts only the combinatorial structure outside a finite part $K,$ and it's challenging to deduce the structure inside  $K.$ Our theorem implies that the combinatorial structure of an asymptotically flat graph is  determined by the non-negativity of the Ricci curvature. This reveals a rather special role of the grid graph in discrete spaces.
		\end{enumerate}
	\end{rem}

	To prove the theorem, we adopt the techniques introduced in \cite{salami,salami2}. In particular, we cut the asymptotically flat graph to obtain $n$ salamis, i.e. two-end graphs with non-negative Ricci curvature, see Section~\ref{sec:pmtg} for details. For each salami, it associates with an integer-valued harmonic function of linear growth. Then we assign coordinates to vertices in $K$ using these harmonic functions. The non-negativity of the Ricci curvature determines the combinatorial structure in $K$ by induction, layer-by-layer from the outside of $K$ to the inside of $K$, which also implies that the weights are trivial.

%\subsection{Structure of the paper}

The paper is organized as follows:
in next section, we give the explicit curvature formulas for weighted grid graphs, and prove the positive mass theorem on them, Theorem~\ref{thm:main1}.
In Section~\ref{sec:pmtg}, we prove the positive mass theorem for asymptotically flat graphs with nonnegative Ricci curvatures, Theorem~\ref{thm:mainric}. Section~\ref{sec:tor} is devoted to the results on discrete tori. In Appendix, we include some counterexamples for the rigidity of the positive mass theorem for weighted graphs with varying vertex weights.
		
	\section{The positive mass theorem on grid graphs}\label{sec:grid}
	\subsection{Explicit formulas for the scalar curvature}
		For the Ricci curvature on grid graphs, we have the following formula.
		\begin{prop}\label{prop:ric}
			Let $\Z^n=(\Z^n,w)$ be a grid graph, $x\in \Z^n$. Assume that $y=x+e_n,$ then
			\begin{align*}
				\kappa(x,y)=&2w(x,y)-w(x,x-e_n)-w(y,y+e_n)\\
				&-\sum_{i=1}^{n-1}(|w(x,x+e_i)-w(y,y+e_i)|+|w(x,x-e_i)-w(y,y-e_i)|).
			\end{align*}
		
		\end{prop}
	
		The proof for the case of $n=2$ is given in \cite[Lemma 6.2]{salami}. The general case holds as follows.
			
		\begin{proof}
			By the definition (\ref{limit-free}) of the Ollivier curvature, we only need to consider $f\in \Lip(1)$ with $f:B_1(x)\cup B_1(y)\rightarrow \{-1,0,1,2\},$ $f(x)=0$ and $f(y)=1.$ By definitions,
			$$
			\begin{aligned}
				\Delta f(x)=&\sum_{i=1}^{n-1} \left(f(x+e_i)w(x,x+e_i)+f(x-e_i)w(x,x-e_i)\right)\\
				&+f(x-e_n)w(x,x-e_n)+w(x,y),
			\end{aligned}
			$$ 
			$$
			\begin{aligned}
				\Delta f(y)=&\sum_{i=1}^{n-1} \left((f(y+e_i)-1)w(y,y+e_i)+(f(y-e_i)-1)w(y,y-e_i)\right)\\
				&+(f(y+e_n)-1)w(y,y+e_n)-w(x,y).
			\end{aligned}
			$$ 
			Note that we need the constraint that $|f(x+e_i)-f(y+e_i)|\le 1$ to guarantee the Lipschitz condition. To attain the infimum in (\ref{limit-free}), we shall require 
			$$f(x-e_n)=-1,\quad f(y+e_n)=2,$$ 
			$$f(y+e_i)=f(x+e_i)+1,\quad f(y-e_i)=f(x-e_i)+1.$$
			Then,
			\begin{align*}
			\nabla_{xy}\Delta f=&2w(x,y)-w(x,x-e_n)-w(y,y+e_n)\\
			&+\sum_{i=1}^{n-1}\left(f(x+e_i)(w(x,x+e_i)-w(y,y+e_i))\right.\\
			&\qquad \left.+f(x-e_i)(w(x,x-e_i)-w(y,y-e_i))\right).
			\end{align*}
			
			As $f(x\pm e_i)\in \{-1,0,1\}$ by the Lipschitz condition, we derive the conclusion by assigning 
			$$f(x+e_i)=-\sgn\left(w(x,x+e_i)-w(y,y+e_i)\right),$$
			$$f(x-e_i)=-\sgn\left(w(x,x-e_i)-w(y,y-e_i)\right)$$
			for $i=1,\cdots,n-1.$
		\end{proof}
		
		Summing up the Ricci curvatures on edges incident to $x,$ we conclude the following corollary.
		
		\begin{cor}\label{cor:scal}
			Let $\Z^n=(\Z^n,w)$ be a grid graph, $x\in \Z^n$. We have the following formula for scalar curvature,
			\begin{small}
				\begin{align*}
					R(x)=&\sum_{i=1}^n(w(x,x_i^+)-w(x_i^+,x+e_i)+w(x,x_i^-)-w(x_i^-,x_i^--e_i))\\
					&-\sum_{\begin{subarray}{c}
							i,j=1\\
							i\neq j
					\end{subarray}}^n\left(|w(x,x+e_i)-w(x_j^+,x^+_j+e_i)|+|w(x,x-e_i)-w(x_j^+,x_j^+-e_i)|\right.\\
					&\qquad+\left.|w(x,x+e_i)-w(x_j^-,x_j^-+e_i)|+|w(x,x-e_i)-w(x_j^-,x^-_j-e_i)|\right),
				\end{align*}
			\end{small}
			where $x_j^+:=x+e_j,$ $x_j^-=x-e_j.$
			
		\end{cor}

		Recalling the definitions in Section \ref{sec:intro}, the scalar curvature of $x$ can be written as
		$$
		R(x)=\sum_{e\in E_0(x)}w(e)-\sum_{\tau\in \wt{E}_0(x)}w(\tau)-\Abs(x),
		$$
		where $\Abs(x)\ge 0$ is the summation of absolute terms in Corollary \ref{cor:scal}. It measures the difference between weights of parallel edges surrounding $x.$ Figure \ref{R(x)} is an illustration for this formula.
				
		\begin{figure}[htbp]
			\centering
			% drawn by GeoGebra
			\definecolor{qqffqq}{rgb}{0,1,0}
			\definecolor{qqqqff}{rgb}{0,0,1}
			\definecolor{ffqqqq}{rgb}{1,0,0}
			\definecolor{ududff}{rgb}{0.3,0.3,1}
			\definecolor{cqcqcq}{rgb}{0.75,0.75,0.75}
			\begin{tikzpicture}[line cap=round,line join=round,>=triangle 45,x=1cm,y=1cm,scale=2]
				\draw [color=cqcqcq,, xstep=0.5cm,ystep=0.5cm] (0.75,0.9) grid (3.25,3.1);
%				\selectcolormodel{cmyk}
				\clip(0.5,0.9) rectangle (3.5,3.1);
				\draw [line width=3pt,color=ffqqqq] (1.5,2)-- (2,2);
				\draw [line width=3pt,color=ffqqqq] (2,2)-- (2,2.5);
				\draw [line width=3pt,color=qqqqff] (2,2.5)-- (2,3);
				\draw [line width=3pt,color=ffqqqq] (2,2)-- (2.5,2);
				\draw [line width=3pt,color=qqqqff] (2.5,2)-- (3,2);
				\draw [line width=3pt,color=qqqqff] (1.5,2)-- (1,2);
				\draw [line width=3pt,color=ffqqqq] (2,2)-- (2,1.5);
				\draw [line width=3pt,color=qqqqff] (2,1.5)-- (2,1);
				\draw [line width=3pt,color=qqffqq] (1.5,2.5)-- (2.5,2.5);
				\draw [line width=3pt,color=qqffqq] (2.5,2.5)-- (2.5,1.5);
				\draw [line width=3pt,color=qqffqq] (2.5,1.5)-- (1.5,1.5);
				\draw [line width=3pt,color=qqffqq] (1.5,1.5)-- (1.5,2.5);
				\begin{scriptsize}
					\draw [fill=ududff] (1.5,2) circle (1pt);
					\draw [fill=ududff] (2,2) circle (1pt);
					\draw[color=black] (2.1,2.1) node {\normalsize $x$};
					\draw [fill=ududff] (2,2.5) circle (1pt);
					\draw [fill=ududff] (2,3) circle (1pt);
					\draw [fill=ududff] (2.5,2) circle (1pt);
					\draw [fill=ududff] (3,2) circle (1pt);
					\draw [fill=ududff] (1,2) circle (1pt);
					\draw [fill=ududff] (2,1.5) circle (1pt);
					\draw [fill=ududff] (2,1) circle (1pt);
					\draw [fill=ududff] (1.5,2.5) circle (1pt);
					\draw [fill=ududff] (2.5,2.5) circle (1pt);
					\draw [fill=ududff] (2.5,1.5) circle (1pt);
					\draw [fill=ududff] (1.5,1.5) circle (1pt);
				\end{scriptsize}
			\end{tikzpicture}
			\caption{Illustration for the scalar curvature formula $R(x)$ when $n=2$. Red edges indicate the plus terms and blue edges indicate the minus terms. Green and red edges are involved in the absolute terms.}
			\label{R(x)}
		\end{figure}
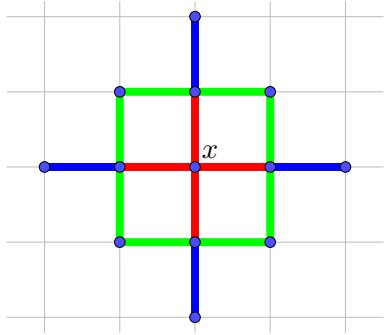
		
		More generally, we have the following expression for the summation of scalar curvatures over a cube.
		\begin{align}
		\sum_{x\in Q_r}R(x)=\sum_{e\in E_r}w(e)-\sum_{\tau \in \wt{E}_r}w(f)-\sum_{x\in Q_r}\Abs(x).\label{key}
		\end{align}		
		
		\subsection{Proof of Theorem~\ref{thm:main1}}
		In this subsection, we prove the positive mass theorem on grid graphs, Theorem~\ref{thm:main1}.
		We first verify that the ADM mass is well-defined and finite. Since the graph is asymptotically flat, both $R(x)$ and $\Abs(x)$ are summable. Then
			\begin{align*}
			\sum_{e\in E_r}w(e)-\sum_{\tau\in \wt{E}_r}w(\tau)=\sum_{x\in Q_r}\Abs(x)+R(x)\ge 0
			\end{align*}
			converges to a non-negative value when $r\rightarrow \infty.$ This proves the first part of the theorem, $E(G)\ge 0$.
			
			For the rigidity part of the theorem, assume $E(G)=0.$ Note that $$\sum_{x\in Q_r} \Abs(x)+R(x)$$ is non-negative, and monotonely increasing with respect to $r,$ whose limit is zero as $r\to\infty.$ This gives $\Abs(x)=R(x)=0$ for all vertices $x\in G.$ As $w(x,y)=1+o(1),$ we conclude that $w\equiv 1$ by the definition of $\Abs(x).$ This finishes the proof of the theorem.
			\qed
	
		We remark that the proof depends on the specific combinatorial structure of the grid graph. In continuous cases, if the topological structure of an asymptotically flat manifold is simple, one may also prove the theorem directly. For instance, \cite{graph} proves the case when the manifold is the graph of a function.
		
		\subsection{Some results and examples}
		In this subsection, we provide several corollaries of the main theorem, which indicate the properness of our definitions for the discrete ADM mass.
		The following corollary is immediate from Theorem~\ref{thm:main1}. It will be extended to general graphs with non-negative Ricci curvature in Theorem~\ref{thm:mainric}.
		\begin{cor}
			If the grid graph $G$ has trivial weights $1$ outside a finite set $K$ with $R\ge 0,$ then $G$ has to be a standard grid graph.\label{cor}
		\end{cor}
		
%If one imposes stronger decay condition on the definition of asymptotically flat graphs, then the mass will be trivial for grid graphs.	
	%	\begin{cor}
	%		For asymptotically flat grid graph $G,$ if $w(x,y)=1+O(|x|^{-p})$ with $|w(x,y)-%w(y,z)|=O(|x|^{-p-1})$ for $x\sim y\sim z$ on the same line, $p>n-2,$ then $E(G)=0$ and $w\equiv 1.$\label{energy-term}
	%	\end{cor}
		In the following, we prove Corollary~\ref{coro:strongdecay}.	
		\begin{proof}[Proof of Corollary~\ref{coro:strongdecay}]
			By the definition of the energy
			$$
			\left|\sum_{e\in E_r}w(e)-\sum_{\tau\in \wt{E}_r}w(\tau)\right|\le 2n(2r+1)^{n-1}Cr^{-p-1}\rightarrow 0.
			$$
			Therefore $E(G)=0,$ and then we have $w\equiv 1.$
		\end{proof}
		
	Note that Corollary \ref{coro:strongdecay} is similar to the result in continuous case \cite[Theorem 3 (ii)]{mass}, where the condition about $|w(x,y)-w(y,z)|$ is to mimic the restriction on $g_{ij,k}.$ The corollary implies that the energy is contained in the term $|x|^{-(n-2)}$. This can be seen more clearly in next two results.
		
		\begin{cor}
			When $n=2,$ there doesn't exist an asymptotically flat graph $G$ with $E(G)>0.$ The only asymptotically flat graph with non-negative scalar curvature in $2$-dimension is the standard grid graph.
		\end{cor}
	
		\begin{proof}
			Denote the average of weights in $E_r$ by
			$$
			\cl w(r)=\frac{1}{\left|E_r\right|}\sum_{e\in E_r} w(e)=1+o(1).
			$$
			
			We first estimate $\sum_{e\in \wt E_{r}} w(e)$ by $\cl w(r+1).$ Assume $(x,x+e_i)\in \wt E_r.$ We take $j\neq i,$ and denote $x+le_j$ by $x_l.$ As $\Abs(x)=o(|x|^{-2}),$ for $k\in \Z_+,$
			\begin{small}
				$$
				\begin{aligned}
					|w(x,x+e_i)-w(x_k,x_k+e_i)|&\le \sum_{l=1}^k |w(x_{l-1},x_{l-1}+e_i)-w(x_l,x_l+e_i)|\le ko(|x|^{-2}).
				\end{aligned}
				$$
			\end{small}

			The differences between $E_{r+1}$ and $\wt E_r$ are only in the corner. For $n=2,$ by the above estimate, we have
			$$
			\begin{aligned}
				\left|\frac{2r+3}{2r+1}\sum_{e\in \wt E_r} w(e)-\sum_{e\in E_{r+1}} w(e)\right|&\le \left|\frac{2}{2r+1}\sum_{e\in \wt{E}_r}w(e)-\sum_{e\in E_{r+1}\setminus \wt{E}_r} w(e)\right|\\
				&\le \frac{4n}{2r+1}\frac{(1+2r+1)(2r+1)}{2}o(r^{-2})\\
				&=o(r^{-1}),
			\end{aligned}
			$$
			$$
			\begin{aligned}
				\sum_{e\in\wt E_r} w(e)&=\frac{2r+1}{2r+3}\cdot 2n(2r+3)\cl w(r+1)+o(r^{-1})\\
				&=4(2r+1)\cl w(r+1)+o(r^{-1}).
			\end{aligned}
			$$
			
			Now we have
			$$
			\begin{aligned}
				\sum_{e\in E_r}w(e)-\sum_{\tau \in \wt{E}_r}w(\tau)=4(2r+1)(\cl w(r)-\cl w(r+1))+o(r^{-1}). 
			\end{aligned}
			$$
			
			Assume $E(G)>0,$ then 
			$$\lim_{r\rightarrow \infty }r(\cl w(r)-\cl w(r+1))=C>0,\quad \cl w(r+1)-\cl w(r)\le -\frac{C}{2r}$$ 
			for $r$ sufficiently large. It contradicts $\cl w(r)=1+o(1),$ as $\frac{1}{r}$ is not summable.
		\end{proof}
		
		\begin{example}
			In continuous cases, the model space for $E(g)=m$ is 
			$$(M^n,g)=\left(\R^n\setminus {B}_{|m|^{\frac{1}{n-2}}}, \left(1+\frac{m}{|x|^{n-2}}\right)^{\frac{4}{n-2}}\delta\right)$$
			with $R\equiv 0,$ which is called the \textit{Schwarzschild manifold}. We try to give some model graph for $E(G)=m.$ 
			
			Consider the grid graph. We only need to define weights for edges in $E_r.$ We can copy the weights to other edges by restricting $\Abs(x)=0$ everywhere. Similar to continuous cases, we define $$w(e)=\left(1+\frac{m}{(r+1)^{n-2}}\right)^{\frac{1}{n-2}},\quad e\in E_r$$
			for $n\ge 3.$ Then 
			$$
			\sum_{e\in E_r} w(e)-\sum_{\tau\in \wt{E}_r} w(\tau)\sim 2n(2r+1)^{n-1}\frac{1}{n-2}\left(\frac{m}{r^{n-2}}-\frac{m}{(r+1)^{n-2}}\right)\rightarrow 2^nnm,
			$$
			so $E(G)=m.$ Unfortunately, scalar curvatures in this graph are not non-negative everywhere. For $x=(x_1,\cdots,x_n)$ with all components sufficiently large,
			\begin{align*}
			R(x)=&\sum_{i=1}^n \left(\left(1+\frac{m}{|x_i|^{n-2}}\right)^{\frac{1}{n-2}}+\left(1+\frac{m}{(|x_i|+1)^{n-2}}\right)^{\frac{1}{n-2}}\right.\\
			&\left.-\left(1+\frac{m}{(|x_i|-1)^{n-2}}\right)^{\frac{1}{n-2}}-\left(1+\frac{m}{(|x_i|+2)^{n-2}}\right)^{\frac{1}{n-2}}\right).
			\end{align*}
			As $\left(1+\frac{m}{(r+1)^{n-2}}\right)^{\frac{1}{n-2}}$ is convex for $r$ large enough, it's easy to see that $R(x)<0$. 
			
			For $n=2,$ we have already shown that $E(G)=0$ for asymptotically flat grid graphs. But if we formally take $w(e)=1-m\log r$, then 
			$$
			\sum_{e\in E_r} w(e)-\sum_{f\in \wt{E}_r} w(f)=4(2r+1)m\log\left(1+\frac{1}{r}\right)\rightarrow 8m,
			$$
			so $E(G)=m.$ However, it violates the conditions $w\ge 0$ and $w=1+o(1)$ when $m\neq 0.$

		\end{example}
	
	\section{The positive mass theorem on general graphs}\label{sec:pmtg}
		In this section, we prove the positive mass theorem on general asymptotically flat graphs under Ricci curvature condition, Theorem~\ref{thm:mainric}.

		First we introduce the results on salamis, which will play important roles in the proof. 
		
		\subsection{Salami}   
			\begin{defi}\textnormal{\cite{salami,salami2}}
				A salami is a connected and locally finite graph $G=(V,E,w)$ equipped with vertex weights $m:V\to (0,\infty),$ which has non-negative Ollivier curvature, and has at least two ends of infinite $m$-volumes. \label{def:salami}
			\end{defi}
			
			For a salami, one can always find a \textit{salami partition} $P=(X,Y,K)$, such that
			\begin{enumerate}[(i)]
				\item $X\sqcup Y\sqcup K=V,$ where $\sqcup$ denotes the disjoint union;
				\item $K$ is finite,
				\item $E(X,Y)=\emp,$
				\item $m(X),m(Y)=\infty.$
			\end{enumerate}
			In addition, we call a salami partition connected if $K$ is connected.
			
			By \cite[Theorem 1.1, Theorem 1.3]{salami}, any salami has exactly two ends, and it's quasi-isometric to the line provided that the weights are non-degenerate. In this paper, we always assume $m\equiv 1.$
				
			We mainly refer to the \textit{extremal Lipschitz extension} method used in that paper. For any $\Omega\subset V,$ we denote by  
			$$\Lip(1,K):=\{f:V\rightarrow \R\:|\:f(y)-f(x)\le d(x,y),\foa x,y\in K\}$$
			the set of $1$-Lipschitz functions on $\Omega.$
			
			\begin{defi}
				Let $P=(X,Y,K)$ be a salami partition. For $f\in \Lip(1,K),$ we set 
				$$
				S(P)f(v):=Sf(v):=\begin{cases}
					f(v),&v\in K,\\
					\sup_{w\in K}f(w)-d(v,w),&v\in X,\\
					\inf_{w\in K}f(w)+d(v,w),&v\in Y.
				\end{cases}
				$$
				$$
				\FF:=\FF(P):=S(P)(\Lip(1,K)), \quad \HH:=\HH(P):=\{f\in \FF(P):\Delta f=0\}.
				$$
				
			\end{defi}
			
			Note that $S(P)f\in \Lip(1,V).$ We have the following lemmas in \cite[Lemma 3.8, 3.11-12]{salami}.
			
			\begin{lemma}
				Let $G$ be a salami with connected salami partition $P = (X,Y,K)$. Suppose
				$\Delta f = 0$ on $K$ for some $f\in \FF$. Then $\Delta f = 0$ everywhere. In particular, $\HH(P)\neq \emp.$\label{laplace}
				
			\end{lemma}
		
		The above lemma yields that for a salami there exists a Lipschitz harmonic function. For the case of non-degenerate weights, the salami is quasi-isometric to the line, and the harmonic function serves as a ``coordinate function'' of the graph along the line direction.
		
			\begin{lemma}
				Let $P_i=(X_i,Y_i,K_i)$ for $i=1,2$ be connected salami partitions with $X_i\setminus X_j$ and $Y_i\setminus Y_j$ finite for $i,j=1,2.$ Then, $\HH(P_1)=\HH(P_2)=h+\R$ for some $h\in \HH(P_i).$ \label{harmonic}
			\end{lemma}
			
		This gives the uniqueness of obtained Lipschitz harmonic function.
		
		\subsection{Proof of Theorem~\ref{thm:mainric}}
		We first prove a crucial lemma, which establishes that the edge weights become trivial outside a finite set in an asymptotically flat graph, under the condition of non-negative Ricci curvatures. 
		
		\begin{lemma}
			Let $G = (V, E, w)$ be an asymptotically flat graph with non-negative Ricci curvatures. Then, there exists a finite set $W \subseteq V$ such that for any edge $e \in E$ with both endpoints outside of $W$, the weight $w(e)=1$. \label{lem:trivial}
			
		\end{lemma}
		
		\begin{proof}
			As $G$ is an asymptotically flat graph, there is a finite set $K$ such that we have the isomorphism $\Phi:G\setminus K\cong \Z^n\setminus Q_{r}$ and $E(K,\Phi^{-1}(\Z^n\setminus Q_{r+1}))=\emp.$ We claim that the weights of edges are equal to $1$ outside $W:=K\sqcup \Phi^{-1}(S_{r+1}).$ For $x\notin W,$ we have $\Phi(x)=(\Phi_1(x),\cdots,\Phi_n(x))\in \Z^n\setminus Q_{r+1}.$ If $|\Phi_i(x)|>r+1,$ then for all $j\neq i,$ the straight line $\{x+ke_j\mid k\in \Z\}$ has no intersection with $W.$ Let $w_k:=w(x+ke_j,x+(k+1)e_j).$ We can calculate the Ollivier curvature of the edge $(x+ke_j,x+(k+1)e_j)$ by Proposition \ref{prop:ric}. Then we have 
			$$2w_k-w_{k+1}-w_{k-1}\ge \kappa(x+ke_j,x+(k+1)e_j)\ge 0.$$
			
			Hence $w_k$ is a discrete concave function with respect to $k.$ As $w_k>0,$ we have $w_k\equiv c,$ and then $c=1$ by the asymptotically flat condition. In this way we figure out the weights of all lines outside $W$.
			
			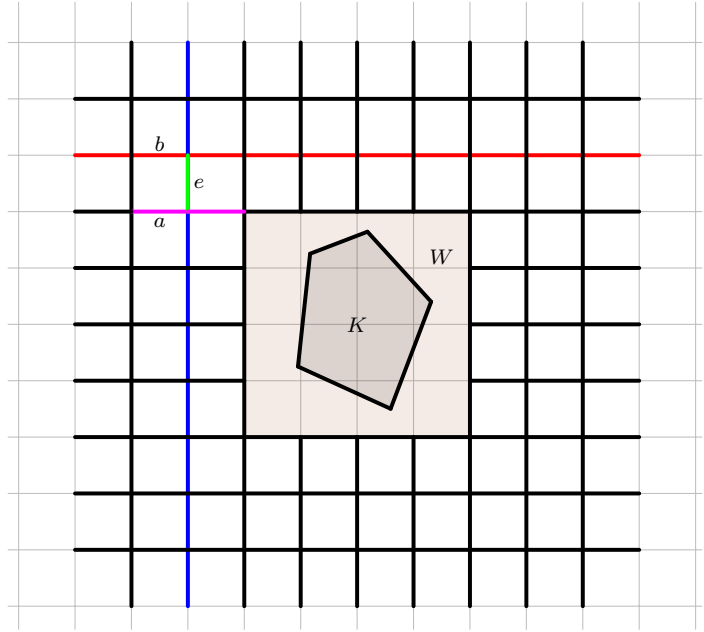
\begin{figure}[htbp]
				\centering
				% drawn by GeoGebra
%				\definecolor{qqffqq}{rgb}{0.8,0.8,0.7}
				\definecolor{ffxfqq}{rgb}{1,0,1}
%				\definecolor{qqqqff}{rgb}{0,0,0}
%				\definecolor{ffqqqq}{rgb}{0,0,0}
				\definecolor{zzttqq}{rgb}{0.6,0.2,0}
%				\definecolor{cqcqcq}{rgb}{0.75,0.75,0.75}
				\definecolor{qqffqq}{rgb}{0,1,0}
				\definecolor{qqqqff}{rgb}{0,0,1}
				\definecolor{ffqqqq}{rgb}{1,0,0}
				\definecolor{ududff}{rgb}{0.3,0.3,1}
				\definecolor{cqcqcq}{rgb}{0.75,0.75,0.75}
				\begin{tikzpicture}[line cap=round,line join=round,>=triangle 45,x=1cm,y=1cm,scale=0.75]
%					\selectcolormodel{gray}
					\draw [color=cqcqcq,, xstep=1cm,ystep=1cm] (-1.1869510185584664,0.590195972481711) grid (11.11190552922686,11.714240927819638);
					\clip(-1.1869510185584664,-0.190195972481711) rectangle (11.11190552922686,11.714240927819638);
					\fill[line width=2pt,color=zzttqq,fill=zzttqq,fill opacity=0.1] (3,8) -- (3,4) -- (7,4) -- (7,8) -- cycle;
					\fill[line width=2pt,fill=black,fill opacity=0.1] (4.166955122113556,7.253055021607624) -- (3.951428712456659,5.251738360507876) -- (5.59353469079492,4.502527507891046) -- (6.311956056317909,6.401212545344654) -- (5.183008196210355,7.64305519146296) -- cycle;
					\begin{scriptsize}
						\draw[color=black] (6.5,7.2) node {$W$};
						\draw[color=black] (5,6) node {$K$};
						\draw[color=black] (2.2,8.5) node {$e$};
						\draw[color=black] (1.5,7.8) node {$a$};
						\draw[color=black] (1.5,9.2) node {$b$};
					\end{scriptsize}
					\draw [line width=1.5pt] (3,8)-- (3,4);
					\draw [line width=1.5pt] (3,4)-- (7,4);
					\draw [line width=1.5pt] (7,4)-- (7,8);
					\draw [line width=1.5pt] (7,8)-- (3,8);
					\draw [line width=1.5pt] (4.166955122113556,7.253055021607624)-- (3.951428712456659,5.251738360507876);
					\draw [line width=1.5pt] (3.951428712456659,5.251738360507876)-- (5.59353469079492,4.502527507891046);
					\draw [line width=1.5pt] (5.59353469079492,4.502527507891046)-- (6.311956056317909,6.401212545344654);
					\draw [line width=1.5pt] (6.311956056317909,6.401212545344654)-- (5.183008196210355,7.64305519146296);
					\draw [line width=1.5pt] (5.183008196210355,7.64305519146296)-- (4.166955122113556,7.253055021607624);
					\draw [line width=1.5pt,color=ffqqqq] (0,9)-- (10,9);
					\draw [line width=1.5pt,color=qqqqff] (2,11)-- (2,1);
					\draw [line width=1.5pt] (0,3)-- (10,3);
					\draw [line width=1.5pt] (8,1)-- (8,11);
					\draw [line width=1.5pt] (1,11)-- (1,1);
					\draw [line width=1.5pt] (0,2)-- (10,2);
					\draw [line width=1.5pt] (9,1)-- (9,11);
					\draw [line width=1.5pt] (10,10)-- (0,10);
					\draw [line width=1.5pt] (3,11)-- (3,8);
					\draw [line width=1.5pt] (4,11)-- (4,8);
					\draw [line width=1.5pt] (5,11)-- (5,8);
					\draw [line width=1.5pt] (6,11)-- (6,8);
					\draw [line width=1.5pt] (7,11)-- (7,8);
					\draw [line width=1.5pt] (10,8)-- (7,8);
					\draw [line width=1.5pt] (10,7)-- (7,7);
					\draw [line width=1.5pt] (10,6)-- (7,6);
					\draw [line width=1.5pt] (10,5)-- (7,5);
					\draw [line width=1.5pt] (10,4)-- (7,4);
					\draw [line width=1.5pt] (7,1)-- (7,4);
					\draw [line width=1.5pt] (6,1)-- (6,4);
					\draw [line width=1.5pt] (5,1)-- (5,4);
					\draw [line width=1.5pt] (4,1)-- (4,4);
					\draw [line width=1.5pt] (3,1)-- (3,4);
					\draw [line width=1.5pt] (0,4)-- (3,4);
					\draw [line width=1.5pt] (0,5)-- (3,5);
					\draw [line width=1.5pt] (0,6)-- (3,6);
					\draw [line width=1.5pt] (0,7)-- (3,7);
					\draw [line width=1.5pt,color=ffxfqq] (1,8)-- (2,8);
					\draw [line width=1.5pt,color=qqffqq] (2,8)-- (2,9);
					\draw [line width=1.5pt,color=ffxfqq] (3,8)-- (2,8);
					\draw [line width=1.5pt] (1,8)-- (0,8);
				\end{tikzpicture}
				\caption{We first show that the weights for edges on lines outside $W$ are $1,$ such as the red and blue lines. Then every edge on them forces parallel edges to have the same weights. For example, the non-negative Ricci curvature of the green edge $e$ forces the pink edge $a$ to have the same weights with edge $b$ on the line above.}\label{fig:salami}
			\end{figure}
			
			We still need to prove $w(x,x\pm e_i)=1$ when $|\Phi_i(x)|> r+1$ and $|\Phi_j(x)|\le r+1,$ $\foa j\neq i.$ Take any $j\neq i,$ then by the former case, 
			$$w(x-e_j,x)=w(x,x+e_j)=w(x+e_j,x+2e_j)=1.$$
			By Proposition~\ref{prop:ric} again, letting $y:=x+e_j,$ we have
			$$
			\kappa(x,y)=-\sum_{k\neq j}(|w(x,x+e_k)-w(y,y+e_k)|-|w(x,x-e_k)-w(y,y-e_k)|)\ge 0,
			$$
			which yields that $w(x,x\pm e_i)=w(y,y\pm e_i)$ respectively. By induction, we have 
			$$w(x,x\pm e_i)=w(x+le_j,x+le_j\pm e_i),\quad \foa l\in \Z.$$
			
			As $|\Phi_j(x+3(r+1)e_j)|>r+1,$ we have $w(x+3(r+1)e_j,x+3(r+1)e_j\pm e_i)=1$ by the former case. This gives us $w(x,x\pm e_i)=1.$
			
 			Now for any $x\notin W,$ there exists $i$ such that $|\Phi_i(x)|>r+1.$ By the first case $w(x,x\pm e_j)=1,$ $\foa j\neq i,$ and by the second case $w(x,x\pm e_i)=1.$ The lemma has been proven.
		\end{proof}
		
%		\begin{lemma}\label{lem:mainric}
%			Assume $G=(V,E,w)$ is an asymptotically flat graph with non-negative Ricci curvatures. If the weights of edges are equal to $1$ outside a finite set, then $G$ is the standard grid graph.
%			
%		\end{lemma}
	
		\begin{proof}[Proof of Theorem~\ref{thm:mainric}]	
			Let $G$ be an asymptotically flat graph $G = (V, E, w)$ with non-negative Ricci curvature. We say $G$ satisfies the property $\mathscr{P}_r$ for $r \in \N$, if the followings hold:
			\begin{enumerate}[(i)]
				\item There exists a finite set $K \subset V$, such that there is a weighted isomorphism $\Phi:G\setminus K\cong \Z^n\setminus Q_r,$ and $E(K,\Phi^{-1}(\Z^n\setminus Q_{r+1}))=\emp;$ 
				\item For every edge $e$ with both endpoints in $V\setminus K,$ the weight $w(e)=1.$
			\end{enumerate}
			
			We further define $\mathscr{P}_{-1}$ to be $(G,w)\cong (\Z^n,1),$ i.e. $G$ is the standard grid graph, which is precisely the conclusion of the theorem. By the definition of the asymptotically flat graph and Lemma~\ref{lem:trivial}, there exists an $r_0\ge 0$ such that $G$ satisfies $A_{r_0}.$ We proceed by claiming that if $G$ satisfies $\mathscr{P}_r,$ then it satisfies $\mathscr{P}_{r-1}$. By induction, we prove the theorem.			
			
			Suppose that $G$ satisfies $A_r$ for $r\in \N,$ we will prove the claim in following steps.

			\textit{Step 1. Construct salamis} 
			
			For $1\le i\le n,$ $s>r,$ consider the induced subgraph on
			$$C_{s}^i:=\{x\in V\setminus K\mid |\Phi_j(x)|\le s,\foa j\neq i,1\le j\le n\}\cup K,$$
			equipped with induced weights from $G$. Letting $R = 4(r+1)$, we assert that $C_R^i$ has non-negative Ricci curvature. In fact, for edges within the induced subgraph on $C_{R-1}^i,$ the local structures are identical to that of $G$, and thus they retain non-negative curvature. Consider an edge $\{x, y\}$ not included in $C_{R-1}^i$. Without loss of generality, assume $\Phi_j(x) = R$ for some $j\neq i.$ If $y = x - e_j$, then $\kappa(x,y) = 1$, following from a similar argument presented in Proposition~\ref{prop:ric}. The computation is straightforward due to the constant weights outside $K$. In the case where $y = x \pm e_k$ with $k \neq j$, then $\kappa(x,y) = 0$, provided that $|\Phi_k(y)| < R$ or $k = i$. Finally, if $y=x\pm e_k$ with $k\neq i$ and $|\Phi_k(y)|=R,$ this scenario is similar to the first case, yielding that $\kappa(x,y) = 1$. Consequently, all edges in $C_R^i$ have non-negative curvatures, and then $C_R^i$ is a salami with two ends of infinite volume in directions of $\pm e_i.$

			Without loss of generality, we consider $C_R:=C_{R}^{1}$. Let 
			$$L_t = C_{R} \cap \{x\in V\setminus K\mid \Phi_1(x)=t\},\quad |t|\in [R,+\infty)\cap \Z.$$
			Consider the connected salami partition $P_t:=(X_t, Y_t, L_t),$ where
			\begin{align*}
				X_t = \{v \in C_R \mid \Phi_1(v) < t\}, \quad Y_t=C_R\setminus (X_t\sqcup L_t), \quad\text{for } &t \le -R, \\
				Y_t = \{v \in C_R \mid \Phi_1(v) > t\},\quad X_t=C_R\setminus (Y_t\sqcup L_t), \quad\text{for } &t\ge R.
			\end{align*}
			
			\textit{Step 2. Make extremal Lipschitz extension}
			
			For $|t|\in [r+2,\infty)\cap \Z,$ define $h_t:=S(P_t)(t\cdot 1_{L_t}),$ where $1_{L_t}$ is the indicator function on $L_t.$ Note that $h_t$ is integer-valued. Obviously $t\cdot 1_{L_t}\in \Lip(1,L_t),$ and $\Delta h_t=0$ within $L_t.$ By Lemma~\ref{laplace}, we conclude that $\Delta h_t=0$ throughout $C_R$, which implies that $h_t \in \HH(P_t).$ Moreover, according to Lemma~\ref{harmonic}, for any two indices $t_1, t_2$ satisfying $|t_i|\ge R,$ $\HH(P_{t_1})=\HH(P_{t_2})=h+\R$ for some $h\in \HH(P_{t_i}),$ $i=1,2.$ Consequently, $h_{t_1}-h_{t_2}=\const,$ which further indicates that $h_{t_1}=\const$ on $L_{t_2}$ too. 
			
			Consider $t=\pm R,$ and take $h_1^\pm:=h_{\pm R}.$ We claim that $h_1^+ = h_{1}^-$. Let $x = \Phi^{-1}(R, \cdots, R)\in L_{R}$, we will show that the distance $d(x, L_{-R}) = 2R$. In fact, choose a shortest path $\gamma$ from $x$ to $L_{-R}$ with the length denoted by $l(\gamma).$ If $\gamma \cap K=\emp,$ then $\Phi(\gamma)\subset \Z^n\setminus Q_r,$ one can easily show that $l(\gamma)=l(\Phi(\gamma))\ge 2R$ because it lies on the grid structure, and the equality holds when $\gamma$ is the segment connecting $\Phi^{-1}(-R,R,\cdots,R)$ directly to $x$.
			
			If $\gamma\cap K\neq \emp,$ then,
			$$
			\begin{aligned}
				l(\gamma)&\ge d(x,K)+d(K,L_{-R})\\
				&\ge n\cdot 3(r+1)+1+3(r+1)+1\\
				&> 2R.
			\end{aligned}
			$$			
			
			Therefore, $d(x,L_{-R})=2R.$ By definition, 
			$$h_1^-(x)=\inf_{w\in L_{-R}} (h_1^-(w)+d(w,x))=-R+2R=R=h_1^+(x),$$
			which implies that $h_{1}^- = h_1^{+}$. From now on, we write $h_1 := h_1^\pm$ for simplicity. 
			
			Next we aim to determine the value of $h_1$ on the set $C_R\setminus K.$
			
			\begin{enumerate} [$\bullet$]
				\item 	For $x$ with $|\Phi_1(x)|\ge r+1,$ we have $x\in L_t$ with $t=\Phi_1(x).$ Without loss of generality, we assume $t>0.$ By the definition of $h_1^+$ and $d(x,L_R)=|R-t|,$ we obtain $h_1(x) = t=\Phi_1(x).$
				\item 	For $x$ with $|\Phi_1(x)|\le r,$ there exists some $j=2,\cdots,n$ such that $|\Phi_j(x)|\ge r+1.$ Therefore, the path $\{v_k\}_{k=-R}^R$ with $v_k=\Phi^{-1}(k,\Phi_2(x),\cdots,\Phi_n(x))$ is included in $C_R\setminus K$. Along this path, we observe that
				$$
				R=h_1(v_R)\le h_1(v_{R-1})+1\le \cdots\le h_1(v_{-R})+2R=R
				$$
				by the Lipschitz condition. Thus, all inequalities are equalities, and we have $h_1(v_{k})=k$ for each $k.$ In particular, $h_1(x)=\Phi_1(x).$
			\end{enumerate}
			
			This shows that $h_1(x) = \Phi_1(x)$ for any $x \in C_R \setminus K$. As a result, $h_1$ naturally extends on $G$ such that $h_1(x)=\Phi_1(x)$ for all $x\in V\setminus K.$ Moreover, note that $h_1$ is an integer-valued harmonic function satisfying $h_1(\delta K)\subset [-r-1,r+1].$ By the maximum principle on $K$, we have $h_1(K)\subset [-r,r].$ 
			
			\begin{figure}[htbp]
				\centering
				% drawn by GeoGebra
				\definecolor{ccqqqq}{rgb}{0.8,0,0}
				\definecolor{ududff}{rgb}{0,0.1,1}
				\definecolor{zzttqq}{rgb}{0.6,0.2,0}
				\definecolor{xdxdff}{rgb}{0,0.1,1}
				\definecolor{cqcqcq}{rgb}{0.7529411764705882,0.7529411764705882,0.7529411764705882}
				\begin{tikzpicture}[line cap=round,line join=round,>=triangle 45,x=1cm,y=1cm]
%					\selectcolormodel{gray}
					\draw [color=cqcqcq,, xstep=0.5cm,ystep=0.5cm] (-0.25,2.25) grid (9.25,8.75);
					\clip(-0.25,2.25) rectangle (9.25,8.75);
					\fill[line width=2pt,color=zzttqq,fill=zzttqq,fill opacity=0.10000000149011612] (4.1,5.1) -- (4.1,5.9) -- (4.9,5.9) -- (4.9,5.1) -- cycle;
					\draw [line width=2pt] (1.5,7)-- (7.5,7);
					\draw [line width=2pt] (7.5,6)-- (1.5,6);
					\draw [line width=2pt] (1.5,5)-- (7.5,5);
					\draw [line width=2pt] (7.5,4)-- (1.5,4);
					\draw [line width=2pt] (1.5,3)-- (7.5,3);
					\draw [line width=2pt] (3,3)-- (3,8);
					\draw [line width=2pt] (4,8)-- (4,3);
					\draw [line width=2pt] (5,3)-- (5,8);
					\draw [line width=2pt] (6,8)-- (6,3);
					\draw [line width=2pt] (2,8)-- (2,8);
					\draw [line width=2pt] (2,8)-- (2,7);
					\draw [line width=2pt] (2,7)-- (2,6);
					\draw [line width=2pt] (2,6)-- (2,5);
					\draw [line width=2pt] (2,5)-- (2,4);
					\draw [line width=2pt] (2,4)-- (2,3);
					\draw [line width=2pt] (2,3)-- (2,3);
					\draw [line width=2pt] (7,8)-- (7,8);
					\draw [line width=2pt] (7,8)-- (7,7);
					\draw [line width=2pt] (7,7)-- (7,6);
					\draw [line width=2pt] (7,6)-- (7,5);
					\draw [line width=2pt] (7,5)-- (7,4);
					\draw [line width=2pt] (7,4)-- (7,3);
					\draw [line width=2pt] (7,3)-- (7,3);
					\draw [line width=2pt,color=ccqqqq] (7,8)-- (2,8);
					\draw [line width=2pt] (2,8)-- (1.5,8);
					\draw [line width=2pt] (7,8)-- (7.5,8);
					\draw [line width=2pt] (2.5,8)-- (2.5,8);
					\draw [line width=2pt] (2.5,8)-- (2.5,7);
					\draw [line width=2pt] (2.5,7)-- (2.5,6);
					\draw [line width=2pt] (2.5,6)-- (2.5,5);
					\draw [line width=2pt] (2.5,5)-- (2.5,4);
					\draw [line width=2pt] (2.5,4)-- (2.5,3);
					\draw [line width=2pt] (2.5,3)-- (2.5,3);
					\draw [line width=2pt] (6.5,8)-- (6.5,8);
					\draw [line width=2pt] (6.5,8)-- (6.5,7);
					\draw [line width=2pt] (6.5,7)-- (6.5,6);
					\draw [line width=2pt] (6.5,6)-- (6.5,5);
					\draw [line width=2pt] (6.5,5)-- (6.5,4);
					\draw [line width=2pt] (6.5,4)-- (6.5,3);
					\draw [line width=2pt] (6.5,3)-- (6.5,3);
					\begin{scriptsize}
						\draw [fill=xdxdff] (5,5) circle (2pt);
						\draw [fill=xdxdff] (4,5) circle (2pt);
						\draw [fill=xdxdff] (4,6) circle (2pt);
						\draw [fill=xdxdff] (5,6) circle (2pt);
						\draw [fill=ududff] (2,8) circle (2pt);
						\draw[color=black] (1.6,8.184195749313927) node {$-R$};
						\draw [fill=ududff] (2,7) circle (2pt);
						\draw[color=black] (1.6,7.1848245094937875) node {$-R$};
						\draw [fill=xdxdff] (2,6) circle (2pt);
						\draw[color=black] (1.6,6.179676557304745) node {$-R$};
						\draw [fill=xdxdff] (2,5) circle (2pt);
						\draw[color=black] (1.6,5.1803053174846045) node {$-R$};
						\draw [fill=xdxdff] (2,4) circle (2pt);
						\draw[color=black] (1.6,4.180934077664465) node {$-R$};
						\draw [fill=xdxdff] (2,3) circle (2pt);
						\draw[color=black] (1.6,3.1815628378443246) node {$-R$};
						\draw [fill=ududff] (7,8) circle (2pt);
						\draw[color=black] (2,2.6) node {$L_{-R}$};
						\draw[color=black] (7,2.6) node {$L_{R}$};
						\draw[color=black] (7.2,8.184195749313927) node {$R$};
						\draw [fill=xdxdff] (7,7) circle (2pt);
						\draw[color=black] (7.2,7.1848245094937875) node {$R$};
						\draw [fill=ududff] (7,6) circle (2pt);
						\draw[color=black] (7.2,6.179676557304745) node {$R$};
						\draw [fill=xdxdff] (7,5) circle (2pt);
						\draw[color=black] (7.2,5.1803053174846045) node {$R$};
						\draw [fill=xdxdff] (7,4) circle (2pt);
						\draw[color=black] (7.2,4.180934077664465) node {$R$};
						\draw [fill=xdxdff] (7,3) circle (2pt);
						\draw[color=black] (7.2,3.1815628378443246) node {$R$};
						\draw[color=zzttqq] (4.5,5.5) node {$K$};
						\draw [fill=xdxdff] (2.5,8) circle (2pt);
						\draw [fill=xdxdff] (2.5,7) circle (2pt);
						\draw [fill=xdxdff] (2.5,6) circle (2pt);
						\draw [fill=ududff] (2.5,5) circle (2pt);
						\draw [fill=xdxdff] (2.5,4) circle (2pt);
						\draw [fill=xdxdff] (2.5,3) circle (2pt);
						\draw [fill=xdxdff] (6.5,8) circle (2pt);
						\draw [fill=xdxdff] (6.5,7) circle (2pt);
						\draw [fill=ududff] (6.5,6) circle (2pt);
						\draw [fill=xdxdff] (6.5,5) circle (2pt);
						\draw [fill=xdxdff] (6.5,4) circle (2pt);
						\draw [fill=xdxdff] (6.5,3) circle (2pt);
						\draw [fill=xdxdff] (4,8) circle (2pt);
						\draw [fill=xdxdff] (4,7) circle (2pt);
						\draw [fill=xdxdff] (4,3) circle (2pt);
						\draw [fill=xdxdff] (5,3) circle (2pt);
						\draw [fill=xdxdff] (5,4) circle (2pt);
						\draw [fill=xdxdff] (5,7) circle (2pt);
						\draw [fill=ududff] (4,4) circle (2pt);
						\draw [fill=xdxdff] (5,8) circle (2pt);
						\draw[color=black] (2.4,8.4) node {$-R+1$};
						\draw[color=black] (3.8,8.4) node {$-r-1$};
						\draw[color=black] (5.2,8.4) node {$r+1$};
						\draw[color=black] (6.6,8.4) node {$R-1$};
						\draw[color=black] (3.1009091291173285,8.180934077664465) node {$\cdots$};
						\draw[color=black] (5.9,8.180934077664465) node {$\cdots$};
						\draw[color=black] (4.597077632663085,8.180934077664465) node {$\cdots$};
					\end{scriptsize}
				\end{tikzpicture}
				\caption{Illustration of $C_R$ and $h_1$ for $n=2$. The red line shows a shortest path connecting $L_{-R}$ and $L_R$ regardless of the structure of $K$.}\label{fig:2r}
			\end{figure}
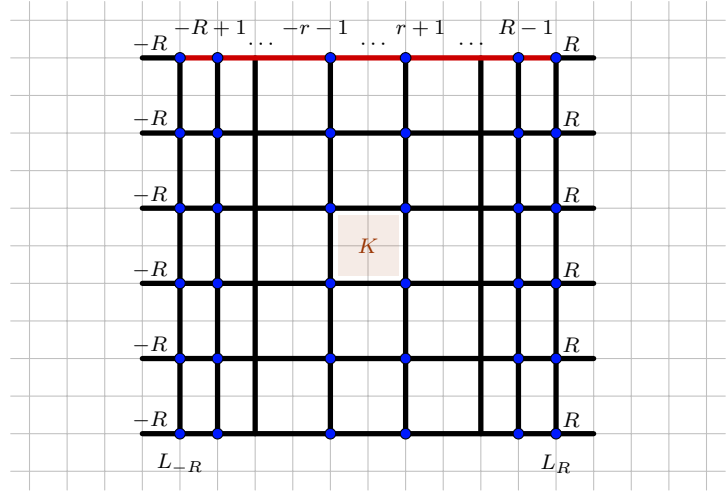
			
			\textit{Step~3. Assign the coordinate in $K$}
			
			Similarly as in Step~2, we can define $h_i(x)$ for all $i=1,\cdots,n,$ and $h_i(x)=\Phi_i(x)$ for all $x\in V\setminus K$ with $|h_i(x)|\le r$ for $x\in K.$ Define the map
			$$
			\wh h=(h_1,\cdots,h_n):G\rightarrow \Z^n,
			$$
			which extends the coordinate function $\Phi$ from $V\setminus K\cong \Z^n\setminus Q_r$ to the entire graph $G.$ Consequently, each vertex $x\in K$ is mapped to a point $\wh h(x)\in \Z^n,$ suggesting a possible grid structure within $K$. However, it remains unclear whether $\widehat{h}$ is bijective on $K$ and how each vertex $x$ connects with other vertices.

			For any $x, y \in \cl K=K\cup \delta K$ with $x\sim y,$ by $h_i\in \Lip(1,V),$ $\foa 1\le i\le n$, we have
			$$
			\|\wh{h}(x) - \wh{h}(y)\|_\infty \leq 1.
			$$
			Our goal is to prove that $x\sim y$ implies $\|\wh{h}(x)-\wh h(y)\|_1\le 1.$ We call an edge $\{x,y\}\in E$ with $x,y\in \cl K$ a \textit{diagonal edge} if $\|\wh h(x)-\wh h(y)\|_1>1.$ This means that we aim to eliminate all diagonal edges, as illustrated in Figure~\ref{K}.

			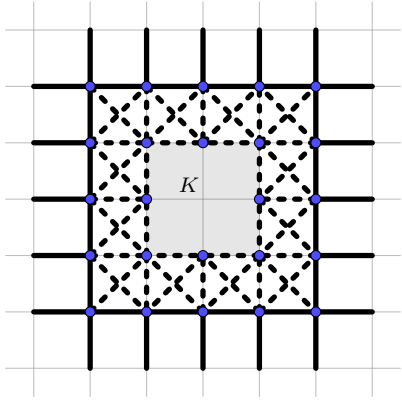
\begin{figure}[htbp]
				\centering
				% drawn by GeoGebra
				\definecolor{xdxdff}{rgb}{0.3,0.3,1}
				\definecolor{ududff}{rgb}{0.3,0.3,1}
				\definecolor{cqcqcq}{rgb}{0.75,0.75,0.75}
				\begin{tikzpicture}[line cap=round,line join=round,>=triangle 45,x=1cm,y=1cm,scale=0.75]
%					\selectcolormodel{gray}
					\draw [color=cqcqcq,, xstep=1cm,ystep=1cm] (1.5,1.5) grid (8.5,8.5);
					\clip(1.5,1.5) rectangle (8.5,8.5);
					\fill[line width=2pt,fill=black,fill opacity=0.1] (4,6) -- (6,6) -- (6,4) -- (4,4) -- cycle;
					\draw [line width=2pt] (3,3)-- (3,7);
					\draw [line width=2pt] (3,7)-- (7,7);
					\draw [line width=2pt] (7,7)-- (7,3);
					\draw [line width=2pt] (7,3)-- (3,3);
					\draw [line width=2pt,dash pattern=on 2pt off 4pt] (4,6)-- (6,6);
					\draw [line width=2pt,dash pattern=on 2pt off 4pt] (6,6)-- (6,4);
					\draw [line width=2pt,dash pattern=on 2pt off 4pt] (6,4)-- (4,4);
					\draw [line width=2pt,dash pattern=on 2pt off 4pt] (4,4)-- (4,6);
					\draw [line width=2pt,dash pattern=on 2pt off 4pt] (3,6)-- (4,6);
					\draw [line width=2pt,dash pattern=on 2pt off 4pt] (3,6)-- (4,5);
					\draw [line width=2pt,dash pattern=on 2pt off 4pt] (3,5)-- (4,5);
					\draw [line width=2pt,dash pattern=on 2pt off 4pt] (3,5)-- (4,4);
					\draw [line width=2pt,dash pattern=on 2pt off 4pt] (3,4)-- (4,4);
					\draw [line width=2pt,dash pattern=on 2pt off 4pt] (4,4)-- (4,3);
					\draw [line width=2pt,dash pattern=on 2pt off 4pt] (5,3)-- (5,4);
					\draw [line width=2pt,dash pattern=on 2pt off 4pt] (6,3)-- (6,4);
					\draw [line width=2pt,dash pattern=on 2pt off 4pt] (7,4)-- (6,4);
					\draw [line width=2pt,dash pattern=on 2pt off 4pt] (7,5)-- (6,5);
					\draw [line width=2pt,dash pattern=on 2pt off 4pt] (7,6)-- (6,6);
					\draw [line width=2pt,dash pattern=on 2pt off 4pt] (6,6)-- (6,7);
					\draw [line width=2pt,dash pattern=on 2pt off 4pt] (5,7)-- (5,6);
					\draw [line width=2pt,dash pattern=on 2pt off 4pt] (4,7)-- (4,6);
					\draw [line width=2pt,dash pattern=on 2pt off 4pt] (4,7)-- (5,6);
					\draw [line width=2pt,dash pattern=on 2pt off 4pt] (4,6)-- (5,7);
					\draw [line width=2pt,dash pattern=on 2pt off 4pt] (5,7)-- (6,6);
					\draw [line width=2pt,dash pattern=on 2pt off 4pt] (5,6)-- (6,7);
					\draw [line width=2pt,dash pattern=on 2pt off 4pt] (6,6)-- (7,5);
					\draw [line width=2pt,dash pattern=on 2pt off 4pt] (7,6)-- (6,5);
					\draw [line width=2pt,dash pattern=on 2pt off 4pt] (6,5)-- (7,4);
					\draw [line width=2pt,dash pattern=on 2pt off 4pt] (7,5)-- (6,4);
					\draw [line width=2pt,dash pattern=on 2pt off 4pt] (6,4)-- (5,3);
					\draw [line width=2pt,dash pattern=on 2pt off 4pt] (6,3)-- (5,4);
					\draw [line width=2pt,dash pattern=on 2pt off 4pt] (5,4)-- (4,3);
					\draw [line width=2pt,dash pattern=on 2pt off 4pt] (5,3)-- (4,4);
					\draw [line width=2pt,dash pattern=on 2pt off 4pt] (3,4)-- (4,5);
					\draw [line width=2pt,dash pattern=on 2pt off 4pt] (3,5)-- (4,6);
					\draw [line width=2pt,dash pattern=on 2pt off 4pt] (3,6)-- (4,7);
					\draw [line width=2pt,dash pattern=on 2pt off 4pt] (3,7)-- (4,6);
					\draw [line width=2pt,dash pattern=on 2pt off 4pt] (3,4)-- (4,3);
					\draw [line width=2pt,dash pattern=on 2pt off 4pt] (3,3)-- (4,4);
					\draw [line width=2pt,dash pattern=on 2pt off 4pt] (6,4)-- (7,3);
					\draw [line width=2pt,dash pattern=on 2pt off 4pt] (6,3)-- (7,4);
					\draw [line width=2pt,dash pattern=on 2pt off 4pt] (7,6)-- (6,7);
					\draw [line width=2pt,dash pattern=on 2pt off 4pt] (6,6)-- (7,7);
					\draw [line width=2pt] (3,7)-- (2,7);
					\draw [line width=2pt] (2,6)-- (3,6);
					\draw [line width=2pt] (3,5)-- (2,5);
					\draw [line width=2pt] (2,4)-- (3,4);
					\draw [line width=2pt] (3,3)-- (2,3);
					\draw [line width=2pt] (3,3)-- (3,2);
					\draw [line width=2pt] (4,3)-- (4,2);
					\draw [line width=2pt] (5,3)-- (5,2);
					\draw [line width=2pt] (6,3)-- (6,2);
					\draw [line width=2pt] (7,3)-- (7,2);
					\draw [line width=2pt] (7,3)-- (8,3);
					\draw [line width=2pt] (8,4)-- (7,4);
					\draw [line width=2pt] (7,5)-- (8,5);
					\draw [line width=2pt] (8,6)-- (7,6);
					\draw [line width=2pt] (7,7)-- (8,7);
					\draw [line width=2pt] (7,7)-- (7,8);
					\draw [line width=2pt] (6,8)-- (6,7);
					\draw [line width=2pt] (5,7)-- (5,8);
					\draw [line width=2pt] (4,8)-- (4,7);
					\draw [line width=2pt] (3,7)-- (3,8);
					\begin{scriptsize}
						\draw [fill=ududff] (3,3) circle (2.5pt);
						\draw [fill=ududff] (3,7) circle (2.5pt);
						\draw [fill=ududff] (7,7) circle (2.5pt);
						\draw [fill=ududff] (7,3) circle (2.5pt);
						\draw [fill=ududff] (3,6) circle (2.5pt);
						\draw [fill=xdxdff] (3,5) circle (2.5pt);
						\draw [fill=xdxdff] (3,4) circle (2.5pt);
						\draw [fill=ududff] (7,6) circle (2.5pt);
						\draw [fill=xdxdff] (7,5) circle (2.5pt);
						\draw [fill=xdxdff] (7,4) circle (2.5pt);
						\draw [fill=ududff] (4,6) circle (2.5pt);
						\draw [fill=ududff] (6,6) circle (2.5pt);
						\draw [fill=ududff] (6,4) circle (2.5pt);
						\draw [fill=ududff] (4,4) circle (2.5pt);
						\draw [fill=ududff] (4,5) circle (2.5pt);
						\draw [fill=ududff] (4,3) circle (2.5pt);
						\draw [fill=xdxdff] (5,3) circle (2.5pt);
						\draw [fill=xdxdff] (5,4) circle (2.5pt);
						\draw [fill=xdxdff] (6,3) circle (2.5pt);
						\draw [fill=xdxdff] (6,5) circle (2.5pt);
						\draw [fill=xdxdff] (6,7) circle (2.5pt);
						\draw [fill=xdxdff] (5,7) circle (2.5pt);
						\draw [fill=xdxdff] (5,6) circle (2.5pt);
						\draw [fill=xdxdff] (4,7) circle (2.5pt);
						\draw[color=black] (4.75,5.25) node {$K$};
						%				\draw[color=black] (2.75,6.25) node {$z$};
						%				\draw[color=black] (3.5,6.25) node {$w_1$};
						%				\draw[color=black] (3.5,5.85) node {$w_2$};
						%				\draw[color=black] (3.25,5.5) node {$w_3$};
					\end{scriptsize}
				\end{tikzpicture}
				\caption{The local structure of $K$ for $n=2$. The dashed lines represent edges to be determined. We will eliminate diagonal edges and keep only the horizontal and vertical edges.}\label{K}
			\end{figure}
			
			\textit{Step~4. Exclude diagonal edges}
			
			Assume that there exists indices $1\le i<j\le n,$ such that $|h_i(x)-h_i(y)|=|h_j(x)-h_j(y)|=1$ and $x\sim y.$ Without loss of generality, we assume
			$$
			h_i(y)=h_i(x)+1,\quad h_j(y)=h_j(x)+1.
			$$
			
			Let $R'=8r+2.$ Consider the induced subgraph on
			$$
			C_{R'}^{ij}=K\cup \{v\in V\setminus K\mid |\Phi_i(v)-\Phi_j(v)|\le R',|\Phi_k(v)|\le R', \foa k\neq i,j\}.
			$$
			It forms a salami again in a similar way to Step~1. Within the induced subgraph on $C_{R'-1}^{ij}$, the curvatures of edges remain $0$. For the remaining edges, those exhibiting a zigzag structure as depicted at the top and bottom of Figure~\ref{diagnoal-salami} also maintain the curvature of $0$, according to similar calculations presented in Proposition~\ref{prop:ric}. Besides, the other edges correspond to the case discussed in Step~1. The curvature of these edges is $1$ if their directions are along $\pm e_k$ with an endpoint $x$ satisfying $|\Phi_k(x)|=R'$ for some $k\neq i,j$. Otherwise, the curvature is $0$.			
			
			Adopting the arguments from Step~1 and Step~2 with the salami partition $P_t'=(X_t',Y_t',L_t'),$ where 
			$$L_t'=C_{R'}^{ij}\cap \{v\in V\setminus K\mid \Phi_i(v)+\Phi_j(v)\in \{t,t\pm 1\}\}$$
			 and $X_t',Y_t'$ are defined similarly. Take $\wt h^\pm:=S(P_{\pm (2r+3)}')(\Phi_i+\Phi_j|_{L'_{\pm (2r+3)}})$ to be the extremal Lipschitz extension on $L'_{\pm (2r+3)}.$ Noting that $\Delta \wt h^\pm =0$, we have $\wt h^+-\wt h^-=\const.$ We aim to further show that $\wt h^+=\wt h^-.$ 
			
			Consider a vertex $v\in L'_{2r+3}$ with $\Phi_i(v)=-3r,$ $\Phi_j(v)=5r+2,$ and $\Phi_k(v)=R'$ for all $k\neq i,j.$ Take the shortest path $\gamma$ from $v$ to $L'_{-2r-3}$. If $\gamma\cap K=\emp,$ then the shortest length $l(\gamma)$ from $v$ to $L'_{-2r-3}$ is $4r+4$. If $\gamma\cap K\neq \emp,$ we have 
			$$
			\begin{aligned}
				l(\gamma)&\ge d(v,K)+d(K,L'_{-2r-3})\\
				&\ge 2r+4r+2+(n-2)(R'-r)+2 \ge 4r+4,
			\end{aligned}
			$$
			indicating that $d(v,L'_{-2r-3})=4r+4$. Therefore, 
			$$\wt h^-(v)=\inf_{w\in L'_{-2r-3}} \wt h^-(w)+d(w,v)=-2r-2+d(v,L'_{-2r-3})=2r+2=\wt h^+(v),$$
			implying that $\wt h^-=\wt h^+,$ and we denote this harmonic function by $h'.$ 
						
			It is straightforward to verify that $\wt h^+=\Phi_i+\Phi_j$ (resp. $\wt h^-=\Phi_i+\Phi_j$) on $L_t'$ for $t\ge 2r+2$ (resp. $t\le -2r-2$). For any other vertex $v$ outside $K,$ there exists a polyline from $L_{-2r-3}'$ to $L_{2r+3}'$ passing through $v$, which forces $h'(v)=\Phi_i(v)+\Phi_j(v)$ to ensure the $1$-Lipschitz condition as in Step~2. This implies that $h'(v)=(\Phi_i+\Phi_j)(v)$ for all $v\in C^{ij}_{R'}\setminus K.$ Furthermore, $h'$ naturally extends to $V$, such that $h'=\Phi_i+\Phi_j$ on $V\setminus K.$			
			
			Now we have
			$$h'|_{V\setminus K}=(\Phi_i+\Phi_j)|_{V\setminus K}=(h_i+h_j)|_{V\setminus K},$$
			where $\{h_i\}_{i=1}^n$ are harmonic functions obtained in Step~3. Since both $h'$ and $h_i+h_j$ are harmonic functions, $h'$ is equal to $h_i+h_j$ in $K$ as well, by virtue of the maximum principle. Now, considering the vertices $x,y$ mentioned earlier,
			$$
			h'(y)-h'(x)=h_i(y)+h_j(y)-h_i(x)-h_j(x)=2,
			$$
			which contradicts the $1$-Lipschitz condition for $h'.$ Thus every diagonal edge is excluded.
			
			\begin{figure}[htbp]
					\centering
					% drawn by GeoGebra
					\definecolor{xdxdff}{rgb}{0.5,0.5,1}
					\definecolor{ududff}{rgb}{0.3,0.3,1}
					\definecolor{cqcqcq}{rgb}{0.75,0.75,0.75}
					\begin{tikzpicture}[line cap=round,line join=round,>=triangle 45,x=1cm,y=1cm,scale=1]
%						\selectcolormodel{gray}
					\clip(7.2,0) rectangle (12.8,4.0);
					\draw [line width=2pt,domain=6.5:14.0] plot(\x,{(-3.5--0.5*\x)/0.5});
					\draw [line width=2pt,domain=6.5:14.0] plot(\x,{(--6.5-0.5*\x)/0.5});
					\draw [line width=2pt,domain=6.5:14.0] plot(\x,{(--4-0.5*\x)/-0.5});
					\draw [line width=2pt,domain=6.5:14.0] plot(\x,{(-6--0.5*\x)/-0.5});
					\draw [line width=2pt,domain=6.5:14.0] plot(\x,{(--5.5-0.5*\x)/0.5});
					\draw [line width=2pt,domain=6.5:14.0] plot(\x,{(--7-0.5*\x)/0.5});
					\draw [line width=2pt,domain=6.5:14.0] plot(\x,{(--13.5-1.5*\x)/-1.5});
					\draw [line width=2pt,domain=6.5:14.0] plot(\x,{(-15--1.5*\x)/1.5});
					\draw [line width=2pt,domain=6.5:14.0] plot(\x,{(--15-1.5*\x)/1.5});
					\draw [line width=2pt,domain=6.5:14.0] plot(\x,{(--22.5-1.5*\x)/1.5});
					\draw [line width=2pt,domain=6.5:14.0] plot(\x,{(--15-2.5*\x)/-2.5});
					\draw [line width=2pt,domain=6.5:14.0] plot(\x,{(-27.5--2.5*\x)/2.5});
					\draw [line width=2pt,domain=6.5:14.0] plot(\x,{(--24-1.5*\x)/1.5});
					\draw [line width=2pt,domain=6.5:14.0] plot(\x,{(--24-2*\x)/-2});
					\draw [line width=2pt,domain=6.5:14.0] plot(\x,{(-13.5--1.5*\x)/-1.5});
					\draw [line width=2pt,domain=6.5:14.0] plot(\x,{(-10--2*\x)/2});
					\draw [line width=2pt,domain=6.5:14.0] plot(\x,{(--12-1.5*\x)/1.5});
					\draw [line width=2pt,domain=6.5:14.0] plot(\x,{(--6-1.5*\x)/-1.5});
					\draw [line width=2pt,dash pattern=on 2pt off 4pt] (10,2)-- (11,2);
					\fill[line width=2pt,fill=black,fill opacity=0.1] (7.8,4.1) -- (9.2,4.1) -- (9.2,-0.1) -- (7.8,-0.1) -- cycle;
					\draw[color=black] (8.5,2) node {$L_t'$};
					\end{tikzpicture}
					\caption{The rotated salami $C_{R'}^{ij}$ when $n=2$. The curvatures of edges with zigzag structure remains $0.$ Any diagonal edge will violate the $1$-Lipschitz condition of the coordinate function.}\label{diagnoal-salami}
				\end{figure}
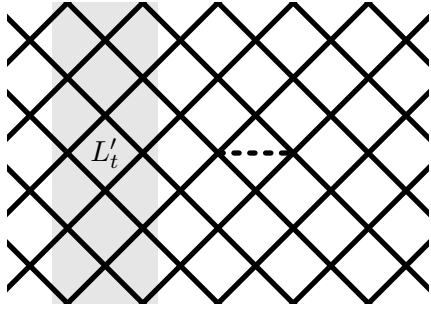
			
			\textit{Step~5. Settle the cross structure}
			
			In order to prove the grid structure within $K$, we first establish the following property: for all $x\in \cl K$ and $1\le i\le n,$ there exist vertices $v_i^{\pm}$ such that $v\sim v_i^\pm$ and
			\begin{align}
				h_i(v)\pm 1=h_i(v_i^{\pm}),\quad h_j(v)=h_j(v_i^{\pm}),\quad \foa j\neq i. \label{eq:h}
			\end{align}
			
			By definitions, we have 
			$$h_i(v)=\inf_{w\in L_{-R}}h_i(w)+d(v,w)=-R+d(v,L_{-R}).$$
			Consider the path $\gamma$ that realizes the shortest distance from $v$ to $L_{-R}.$ Suppose $u\in \gamma$ is adjacent to $v.$ We have $d(u,L_{-R})=d(v,L_{-R})-1,$ which implies that $h_i(u)=h_i(v)-1.$ From the conclusion in Step~4, we know that $h_j(u)=h_j(v)$ for all $j\neq i.$ Therefore, $u$ is precisely the desired vertex $v_i^-.$ Similarly, we can determine the vertex $v_i^+$ by identifying the shortest path towards $L_R.$
			
			Next, we aim to show that for all $a\in S_r\subset \Z^n,$ $\wh h^{-1}(a)\neq \emp.$ Moreover, for any $a\in S_r$ and $b\in S_{r+1}$ satisfying $\|a-b\|_1=1,$ we have $x\sim \Phi^{-1}(b)$ for all $x\in \wh h^{-1}(a).$ Without loss of generality, let $a = (a_1, a_2, \dots, a_n)$ with $a_1 = -r$. For $x = \wh h^{-1}(a-e_1)$, by the property \eqref{eq:h}, there exists $x_1^+\sim x_1$ such that 
			$$h_1(x_1^+)=h_1(x)+1=a_1=-r, \quad h_j(x_1^+)=h_j(x)=a_j,\quad j=2,\cdots,n.$$
			Thus, $x_1^+\in \wh h^{-1}(a),$ and $\wh h^{-1}(a)\neq \emp.$ Furthermore, without loss of generality, suppose $b=a-e_k\in S_{r+1}$ for some $k=1,\cdots,n.$ For all $x\in \wh h^{-1}(a),$ there exists $x_k^-\sim x$ such that
			$$
			h_k(x_k^-)=h_k(x)-1=a_k-1=b_k,\quad h_j(x_k^-)=h_j(x)=a_j=b_j,\quad \foa j\neq k.
			$$ 
			Consequently, $\wh h(x_k^-)=b.$ As $\wh h$ is injective on $\Phi^{-1}(S_{r+1}),$ $x_k^-$ is exactly $\wh h^{-1}(b)=\Phi^{-1}(b).$ This proves the results.	
			
			\textit{Step~6. Solve the multi-vertices problem}
			
			The last problem to address regarding the coordinate function $\widehat{h}$ is to prove that $|\widehat{h}^{-1}(p)| = 1$ for all $p \in S_r$. We start the arguments from the corner $a = (-r, \ldots, -r)$. Let $f$ be a $1$-Lipschitz function, yet to be determined, which will be employed as a test function for the curvature condition. Suppose there are $k$ vertices $v_1, \ldots, v_k$ that share the same coordinate, i.e., $\widehat{h}(v_1) = \cdots = \widehat{h}(v_k) = a$. Although it is possible for $v_s$ and $v_t$ with $s \neq t$ to be adjacent, this will not cause the problem as we will assign same values of $f$ to both vertices. Each vertex $v_t$ is adjacent to vertices $u_i$ where $\Phi(u_i) = a - e_i$ for $i = 1, \ldots, n$. Define the function $f$ such that
			$$
			f(v_1) = \cdots = f(v_k) = 1, \quad f(u_1) = \cdots = f(u_n) = 0.
			$$
			
			\begin{figure}[htbp]
				\centering
				% drawn by GeoGebra
				\definecolor{ffqqqq}{rgb}{1,0,0}
				\definecolor{xdxdff}{rgb}{0.49019607843137253,0.49019607843137253,1}
				\definecolor{sqsqsq}{rgb}{0.12549019607843137,0.12549019607843137,0.12549019607843137}
				\definecolor{ffqqff}{rgb}{1,0,1}
				\definecolor{qqqqff}{rgb}{0,0,1}
				\definecolor{ududff}{rgb}{0.30196078431372547,0.30196078431372547,1}
				\begin{tikzpicture}[line cap=round,line join=round,>=triangle 45,x=1cm,y=1cm,scale=1.5]
					\clip(4.7292962818737142,-0.899762361803137) rectangle (10.672411472594113,2.007446219151739);
%					\selectcolormodel{gray}
					\draw [line width=1pt,dash pattern=on 2pt off 4pt,domain=2.1292962818737142:13.672411472594113] plot(\x,{(-7.556210410820723--1.3000520726406095*\x)/1.8617958205718237});
					\draw [line width=1pt,dash pattern=on 2pt off 4pt,domain=2.1292962818737142:13.672411472594113] plot(\x,{(--6.51627-0*\x)/6.51627});
					\draw [line width=1pt,dash pattern=on 2pt off 4pt,domain=2.1292962818737142:13.672411472594113] plot(\x,{(-9.209460250624243--1.3000520726406093*\x)/1.8617958205718166});
					\draw [line width=1pt,dash pattern=on 2pt off 4pt,domain=2.1292962818737142:13.672411472594113] plot(\x,{(-10.827953093833065--1.3000520726406095*\x)/1.8617958205718228});
					\draw [line width=1pt,dash pattern=on 2pt off 4pt,domain=2.1292962818737142:13.672411472594113] plot(\x,{(--2.4935726493194976-0*\x)/6.51627});
					\draw [line width=1pt,dash pattern=on 2pt off 4pt,domain=2.1292962818737142:13.672411472594113] plot(\x,{(-1.401420023561622-0*\x)/6.51627});
					\draw [line width=1pt,dash pattern=on 2pt off 4pt,domain=2.1292962818737142:13.672411472594113] plot(\x,{(--10.219705656182052-0*\x)/6.51627});
					\draw [line width=1pt,dash pattern=on 2pt off 4pt,domain=2.1292962818737142:13.672411472594113] plot(\x,{(-6.03223582615654--1.3000520726406093*\x)/1.8617958205718166});
					
					\draw [line width=2pt,color=ffqqff] (6.360254018967033,0.3826687122110499)-- (7.606886597334413,0.7190344976776603);
					\draw [line width=2pt,color=qqqqff] (6.360254018967033,0.3826687122110499)-- (7.65588114398433,0.023311935248843865);
					\draw [line width=2pt,color=qqqqff] (7.606886597334413,0.7190344976776603)-- (6.75,-0.22);
					\draw [line width=1pt,dash pattern=on 2pt off 4pt,color=sqsqsq] (6.75,-0.22)-- (7.65588114398433,0.023311935248843865);
					\draw [line width=2pt,color=qqqqff] (7.631933726723283,0.38266871221105)-- (6.360254018967033,0.3826687122110499);
					\draw [line width=2pt,color=qqqqff] (7.606886597334413,0.7190344976776603)-- (8.861146991572284,0.7288334070076437);
					\draw [line width=2pt,color=qqqqff] (7.606886597334413,0.7190344976776603)-- (8.5,1.2);
					\draw [line width=2pt,color=qqqqff] (6.360254018967033,0.3826687122110499)-- (5.188012832066916,0.38266871221105);
					\draw [line width=2pt,color=ffqqqq] (12.96,5)-- (5.504243942093843,-0.21506475691793353);
					\draw [line width=2pt,color=ffqqqq] (5.504243942093843,-0.21506475691793353)-- (11.232483049428257,-0.22);
					\draw [line width=2pt,color=qqqqff] (6.360254018967033,0.3826687122110499)-- (7.24433,1);
					\draw [line width=2pt,color=qqqqff] (6.360254018967033,0.3826687122110499)-- (5.504243942093843,-0.21506475691793353);
					\draw [line width=2pt,color=qqqqff,-{Stealth[scale=0.8]}] (6.4604142057263605,2.0614850758853827)-- (6.773979304285827,2.326055627794933);
					
					\begin{scriptsize}
						\draw [fill=ududff] (14.417126925108885,1.5683367411390337) circle (1pt);
						\draw [fill=qqqqff] (6.360254018967033,0.3826687122110499) circle (1pt);
						\draw[color=black] (6.28403383778666,0.5455426777027873) node {$u_1$};
						\draw [fill=ududff] (7.606886597334413,0.7190344976776603) circle (1pt);
						\draw[color=black] (7.557892050684496,0.8493088669322716) node {$v_1$};
						\draw [fill=ududff] (7.65588114398433,0.023311935248843865) circle (1pt);
						\draw[color=black] (7.65588114398433,0.1143906671835194) node {$v_3$};
						\draw [fill=qqqqff] (6.75,-0.22) circle (1pt);
						\draw[color=black] (6.7,-0.05) node {$u_2$};
						\draw [fill=qqqqff] (7.631933726723283,0.38266871221105) circle (1pt);
						\draw[color=black] (7.616685506664396,0.4867492217228871) node {$v_2$};
						\draw [fill=ududff] (8.861146991572284,0.7288334070076437) circle (1pt);
						\draw[color=black] (8.93953826621215,0.8095099576022882) node {$2$};
						\draw [fill=ududff] (8.5,1.2) circle (1pt);
						\draw[color=black] (8.5,1.35) node {$0$};
						\draw [fill=qqqqff] (7.24433,1) circle (1pt);
						\draw[color=black] (7.224729133465063,1.182471784151706) node {$-1$};
						\draw [fill=qqqqff] (5.504243942093843,-0.21506475691793353) circle (1pt);
						\draw[color=black] (5.431528726078107,-0.07720593656375235) node {$-1$};
						\draw [fill=qqqqff] (5.188012832066916,0.38266871221105) circle (1pt);
						\draw[color=black] (5.14736035550859,0.5553415870327707) node {$-1$};
						\draw [fill=xdxdff] (4.472946661526562,-0.9351975192196172) circle (1pt);
						\draw [fill=xdxdff] (11.232483049428257,1) circle (1pt);
					\end{scriptsize}
				\end{tikzpicture}
				\caption{Illustration for three vertices $v_1,v_2,v_3$ sharing the same coordinates for $n=2$. Red edges indicate $\Phi^{-1}(S_{r+1}).$ We focus on the curvature of the pink edge $\{u_1,v_1\}$. The weights of blue edges are involved in the calculation. The numbers are the values of $f.$}
			\end{figure}
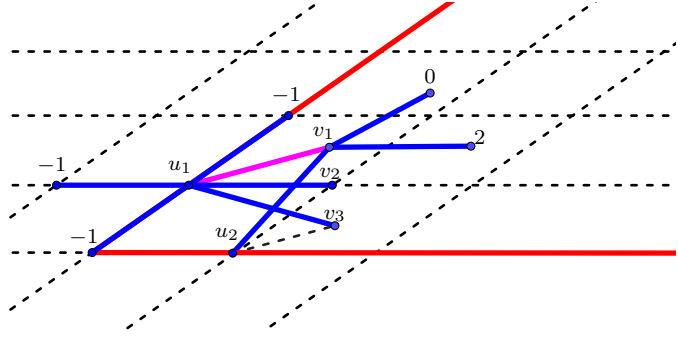
					
			Because $\Delta h_i(u_i)=0,$ we have
			$$
			\sum_{t=1}^k w(u_i, v_t) = w(u_i, u_i - e_i) = 1.
			$$
			If $k > 1$, then $w(u_i, v_t) < 1$, since the weights of edges are strictly positive.
			
			There might be multiple vertices adjacent to $v_t$ within $K$ that share the same coordinates. Henceforth, we fix some $s = 1, \cdots, k$ and $j = 1, \cdots, n$. Let $f$ take the value $2$ on vertices in $\wh h^{-1}(a+e_i)$ for $i = j$, and $0$ for $i \neq j$. Given that $\Delta h_j(v_s) = 0$, we obtain
			$$
			\sum_{z\in \wh h^{-1}(v_s + e_i)} w(v_s, z) = w(v_s, u_i), \quad i = 1, \ldots, n.
			$$
			
			For the function $f$, we assign the value $-1$ to the other neighbors of $u_j$ that have not yet been assigned values. Now with $f(u_j) = 0$ and $f(v_s) = 1$, we can estimate an upper bound of $\kappa(u_j, v_s)$ using $f$,
			$$
			\begin{aligned}
				\Delta f(v_s)&=\sum_{i=1}^n\sum_{z\in \wh h^{-1}(a+e_i)}w(v_s,z)(f(z)-f(v_s))\\
							&\quad +\sum_{i=1}^nw(v_s,u_i)(f(u_i)-f(v_s))\\
							&=w(v_s,u_j)-\sum_{i\neq j}w(v_s,u_i)-\sum_{i=1}^n w(v_s,u_i)\\
							&=-2\sum_{i\neq j}w(v_s,u_i),
			\end{aligned}
			$$
			$$
			\begin{aligned}
				\Delta f(u_j)&=-\sum_{i=1}^n
				w(u_j,u_j-e_i)-\sum_{i\neq j}
				w(u_j,u_j+e_i)\\
				&\quad +\sum_{t=1}^kw(u_j,v_t)(f(v_t)-f(u_j))\\
				&=-n-(n-1)+1\\
				&=-2(n-1),
			\end{aligned}
			$$
			$$\kappa(u_j,v_s)=\inf\limits_{\begin{subarray}{c}
			g:B_1(u_j)\cup B_1(v_s)\rightarrow\Z\\
			g\in \Lip(1)\\
			\nabla_{v_su_j}g=1
			\end{subarray}}\nabla_{u_jv_s}\Delta g\le \nabla_{u_jv_s}\Delta f=\sum_{i\neq j}2(w(v_s,u_i)-1)<0.$$
			This violates the non-negativity of the Ricci curvature. Therefore, the index $k$ must be $1$, which indicates that there is a unique vertex $v\in \wh h^{-1}(a).$ Furthermore, we can infer that $w(u_i, v) = 1$ by $\Delta h_i(u_i) = 0$, for each $i = 1, \ldots, n$.
			
			In conclusion, if $|\wh h^{-1}(a-e_i)|=1$ for all $i=1,\cdots,n$, then $|\wh h^{-1}(a)|=1$ as well. Consequently, we can resolve the structure of $\wh h^{-1}(a+e_k)$ for any $k=1,\cdots,n$ in the same manner, and so forth for all $b=a+\sum_{k=1}^n c_ke_k$ with $c_k\ge 0$ under the restriction that $b\in S_r$ by induction in the lexicographical order of $(h_1, \cdots, h_n)$. 
			
			By symmetry, similar method applies for $a' = (\pm r, \cdots, \pm r)$ as well, which establishes that $|\widehat{h}^{-1}(p)| = 1$ for all $p \in S_r$.
			
			\textit{Step~7. The property $\mathscr{P}_{r-1}$ holds}
			
			Denote $\widehat{h}^{-1}(Q_{r-1})$ by $K'$, we have demonstrated in the preceding steps that $\widehat{h}$ is a weighted isomorphism on $V \setminus K'$. Consequently, the weighted isomorphism $\Phi: V \setminus K \cong \mathbb{Z}^n \setminus Q_r$ can be extended to $\widehat{h}: V \setminus K' \cong \mathbb{Z}^n \setminus Q_{r-1}$. By Step~3, we have ensured that $E(K', \widehat{h}^{-1}(\mathbb{Z}^n \setminus Q_r)) = \emp$. Therefore, the property $\mathscr{P}_{r-1}$ is satisfied. We complete the proof of the theorem.
			
		\end{proof}

		\begin{rem}		
			As demonstrated in the proof of Lemma~\ref{lem:trivial}, the edge weights are $1$ due to the asymptotic assumption that $w(x, y) = 1 + o(1)$ for all adjacent $x \sim y$ as $x \to \infty$. If one drops this assumption, i.e. if we only assume non-negative Ricci curvature and the existence of an isomorphism to a grid graph outside a finite set, similar results can still be obtained by adapting the arguments accordingly. Following analogous proof of Lemma~\ref{lem:trivial}, one can show that the weights of edges in the same direction are constant outside a finite set $W,$ although these constants may vary for different directions. Furthermore, by the proof of Theorem~\ref{thm:mainric}, it can be proved that the graph is also isomorphic to a grid graph with constant weights for edges in the same direction.
		
		\end{rem}
				
		\section{The scalar curvature on discrete tori}\label{sec:tor}
		In this section, we establish the non-existence of positive scalar curvature on discrete tori, as stated in Theorem~\ref{thm:tori}. In the continuous setting, we have the following theorem about curvatures on tori:
		\begin{thm}
			The only possible metric with non-negative scalar curvature on a $n$-torus is the flat metric. 
			
		\end{thm}
		
		The theorem is firstly proved by Scheon-Yau for $n\le 7$ in \cite{SY-torus}, and proved for general cases by Gromov-Lawson in \cite{Gromov1980SpinAS} using the other approach. We will prove a discrete analog of the above result on a discrete torus using similar methods for the positive mass theorem on grid graphs.
		
		\subsection{Discrete tori}
		
		We define discrete tori following \cite{tori}. Consider an $n\times n$ integer matrix $A=[\alpha_1,\cdots,\alpha_n]$ with $\det A\neq 0,$ where $\alpha_i$ is a column vector, $i=1,\cdots,n.$ Take $k\in \Z_+$ and $q=k|\det A|.$ The $n$-dimensional torus $T_A$ with $k^n|\det A|$ vertices is defined to be
		$$
		T_A:=A\Z^n/q\Z^n.
		$$
		
		\begin{example}
			For $n=2,$ let $\alpha_1=(2,-1)^T$, $\alpha_2=(1,3)^T$, $k=1$ and $q=7$. The resulting torus $T_A$ is shown in Figure \ref{torus}.\label{e-torus}

			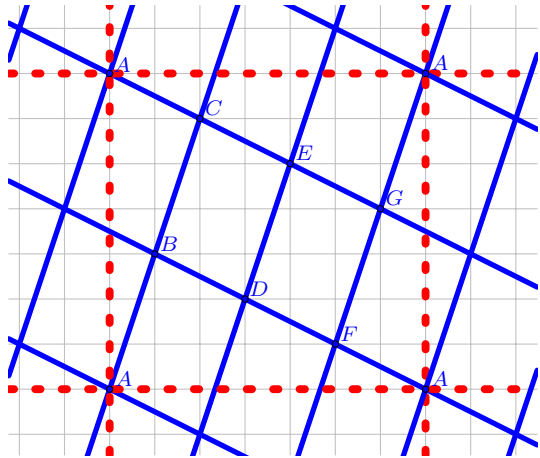
\begin{figure}[htbp]
				\centering
				% drawn by GeoGebra
				\definecolor{ffqqqq}{rgb}{1,0,0}
				\definecolor{qqqqff}{rgb}{0,0,1}
				\definecolor{cqcqcq}{rgb}{0.7529411764705882,0.7529411764705882,0.7529411764705882}
				\begin{tikzpicture}[line cap=round,line join=round,>=triangle 45,x=1cm,y=1cm,scale=0.6]
%					\selectcolormodel{gray}
					\draw [color=cqcqcq,, xstep=1cm,ystep=1cm] (-2.2342648359592334,-1.4602253362572672) grid (9.473317215232584,8.509835703803578);
					\clip(-2.2342648359592334,-1.4602253362572672) rectangle (9.473317215232584,8.509835703803578);
					\draw [line width=2pt,color=qqqqff,domain=-2.2342648359592334:9.473317215232584] plot(\x,{(-0--3*\x)/1});
					\draw [line width=2pt,color=qqqqff,domain=-2.2342648359592334:9.473317215232584] plot(\x,{(-0-1*\x)/2});
					\draw [line width=2pt,color=qqqqff,domain=-2.2342648359592334:9.473317215232584] plot(\x,{(--7-1*\x)/2});
					\draw [line width=2pt,color=qqqqff,domain=-2.2342648359592334:9.473317215232584] plot(\x,{(--14-1*\x)/2});
					\draw [line width=2pt,color=qqqqff,domain=-2.2342648359592334:9.473317215232584] plot(\x,{(--7--3*\x)/1});
					\draw [line width=2pt,color=qqqqff,domain=-2.2342648359592334:9.473317215232584] plot(\x,{(-7--3*\x)/1});
					\draw [line width=2pt,color=qqqqff,domain=-2.2342648359592334:9.473317215232584] plot(\x,{(-14--3*\x)/1});
					\draw [line width=2pt,color=qqqqff,domain=-2.2342648359592334:9.473317215232584] plot(\x,{(-7-1*\x)/2});
					\draw [line width=3pt,dash pattern=on 2pt off 8pt,color=ffqqqq,domain=-2.2342648359592334:9.473317215232584] plot(\x,{(--49-0*\x)/7});
					\draw [line width=3pt,dash pattern=on 2pt off 8pt,color=ffqqqq] (7,-1.4602253362572672) -- (7,8.509835703803578);
					\draw [line width=3pt,dash pattern=on 2pt off 8pt,color=ffqqqq,domain=-2.2342648359592334:9.473317215232584] plot(\x,{(-0-0*\x)/-7});
					\draw [line width=3pt,dash pattern=on 2pt off 8pt,color=ffqqqq] (0,-1.4602253362572672) -- (0,8.509835703803578);
					\draw [line width=2pt,color=qqqqff,domain=-2.2342648359592334:9.473317215232584] plot(\x,{(--63-3*\x)/6});
					\draw [line width=2pt,color=qqqqff,domain=-2.2342648359592334:9.473317215232584] plot(\x,{(-56-12*\x)/-4});
					\draw [line width=2pt,color=qqqqff,domain=-2.2342648359592334:9.473317215232584] plot(\x,{(-21--3*\x)/1});
					\draw [line width=2pt,color=qqqqff,domain=-2.2342648359592334:9.473317215232584] plot(\x,{(-28--3*\x)/1});
					\draw [line width=2pt,color=qqqqff,domain=-2.2342648359592334:9.473317215232584] plot(\x,{(--28-1*\x)/2});
					\begin{scriptsize}
						\draw [fill=qqqqff] (0,0) circle (2pt);
						\draw[color=qqqqff] (0.31345700646174059,0.23075493350408968) node {$A$};
						\draw [fill=qqqqff] (1,3) circle (2pt);
						\draw[color=qqqqff] (1.3166685426504386,3.230047155099157) node {$B$};
						\draw [fill=qqqqff] (2,6) circle (2pt);
						\draw[color=qqqqff] (2.3095376918681192,6.229339376694225) node {$C$};
						\draw [fill=qqqqff] (0,7) circle (2pt);
						\draw[color=qqqqff] (0.28242984554868807,7.201523751969867) node {$A$};
						\draw [fill=qqqqff] (4,5) circle (2pt);
						\draw[color=qqqqff] (4.315960764245516,5.236470227476547) node {$E$};
						\draw [fill=qqqqff] (3,2) circle (2pt);
						\draw[color=qqqqff] (3.312749228056818,2.2371780058814794) node {$D$};
						\draw [fill=qqqqff] (5,1) circle (2pt);
						\draw[color=qqqqff] (5.308829913463197,1.2339664696927848) node {$F$};
						\draw [fill=qqqqff] (6,4) circle (2pt);
						\draw[color=qqqqff] (6.312041449651895,4.233258691287852) node {$G$};
						\draw [fill=qqqqff] (7,0) circle (2pt);
						\draw[color=qqqqff] (7.315252985840593,0.23075493350408968) node {$A$};
						\draw [fill=qqqqff] (7,7) circle (2pt);
						\draw[color=qqqqff] (7.315252985840593,7.232550912882919) node {$A$};
					\end{scriptsize}
				\end{tikzpicture}
				\caption{The torus $T_A$ in Example \ref{e-torus}.\label{torus}}
			\end{figure}
		\end{example}
		
		Note that $x\sim y\in A\Z^n$ if and only if $x-y=\pm\alpha_i,$ $i=1,\cdots,n.$ The weighted graph structure on $T_A$ is induced by the quotient. That is, $[x]\sim [y]\in T_A$ if and only if $x-y\in q\Z^n\pm \alpha_i,$ and for $x\sim y\in A\Z^n,$ we have $w([x],[y])=w(x,y).$ On tori, we have the following formula for the Ollivier Ricci curvature.
		
		\begin{prop}\label{TorusOllivier}
			Let $T_A=A\Z^n/q\Z^n$ be a discrete torus, $[x]\in T_A$. Suppose $y=x+\alpha_n.$ If $\foa u,v\in B_1(\{x,y\})\subset A\Z^n,$ $d(u,v)=d([u],[v]),$ then
			\begin{small}
				\begin{equation}
					\begin{aligned}
						\kappa([x],[y])=&\kappa(x,y)\\
						=&2w(x,y)-w(x,x-\alpha_n)-w(y,y+\alpha_n)\\
						-&\sum_{i=1}^{n-1}(|w(x,x+\alpha_i)-w(y,y+\alpha_i)|+|w(x,x-\alpha_i)-w(y,y-\alpha_i)|).
					\end{aligned}\label{TO}
				\end{equation}
			\end{small}

		\end{prop}
		
		\begin{proof}
			The proof is the same as Proposition \ref{prop:ric}. Restrictions about the distances ensure the local structure of a discrete torus is same as the grid graph.
		\end{proof}
		
		The distance requirement in the above proposition holds for all $y\sim x$ if 
		\begin{align}
			d(u,v)=d([u],[v]),\quad \foa u,v\in B_2(x)\subset A\Z^n.\label{distance}
		\end{align}
		 
		We call (\ref{distance}) the \textit{distance condition}. This is to say that every ball with radius $2$ can be isometrically embedded into the torus. It is automatically true if the torus is large enough. 
		
		The scalar curvature on a discrete torus is defined to be 
		$$R([x])=\sum_{[y]\sim [x]}\kappa([x],[y]),$$
		and the total curvature of a discrete torus is $\sum_{[x]\in T_A}R([x]).$
		
		\begin{proof}[Proof of Theorem \ref{thm:tori}]
			We first claim that
			$$
			S_i([x]):=\sum_{j=0}^{q-1}\kappa([x+j\alpha_i],[x+(j+1)\alpha_i])\le 0.
			$$
			Note that $[x]=[x+q\alpha_i],$ hence this is the sum of curvatures in a cycle on tori, where the edges may be covered multiple times. Plug the formula (\ref{TO}) into the summation. Terms preceding the absolute terms in (\ref{TO}) cancel out in the summation. The remaining terms consist of absolute terms with minus sign. Thus, we have $S_i([x])\le 0.$
			
			Assume curvatures of edges are added $s_i$ times in the summation. Take $r_i\in \Z_+$ to be the minimal positive integer such that $[x]=[x+r_i\alpha_i],$ then $s_i=\frac{q}{r_i}.$ The sum of Ricci curvatures over all edges in the direction of $\alpha_i$ can be calculated as
			$$\frac{1}{s_i}\sum_{[x]\in T_A}\frac{S_i([x])}{r_i}=\frac{\sum_{[x]\in T_A} S_i([x])}{q}\le 0.$$
			
			Note that 
			$$\sum_{[x]\in T_A}R([x])=2\sum_{i=1}^n\frac{\sum_{[x]\in T_A} S_i([x])}{q},$$ 
			because $R([x])$ sums up the curvatures of edges incident to $[x]$, and each edge has two ends, contributing twice in the summation of scalar curvatures. The equality gives us $\sum_{[x]\in T_A} R([x])\le 0.$ This gives the conclusion. Furthermore, if the scalar curvatures on the torus are non-negative, then the torus has to be scalar-flat.
		\end{proof}
		
		\subsection{Other results}
		We have the following theorem in \cite[Corollary B]{Gromov1980SpinAS}: 
		\begin{thm} 
			Let $X=T^n\# X_0$ where $X_0$ is a spin manifold, $n\ge 2$. Then $X$ carries no metric of positive scalar curvature. Any metric of non-negative scalar curvature on $X$ is flat, and $X$ must be the standard torus.
			
		\end{thm}
		
		The theorem has a strong connection to the positive mass theorem. This relationship is more apparent in the discrete case. If the structure on a discrete torus has only been changed locally, we can create a graph that satisfies the conditions in Theorem~\ref{thm:mainric} by opening up the torus. This process is analogous to the inverse approach of constructing a salami. Applying Theorem~\ref{thm:mainric}, we obtain the following corollary.
		
		\begin{cor}			
			Let $T_A$ be a discrete torus with $w\equiv 1$. Suppose the natural map $Q_{r}\hookrightarrow T_A$ is an embedding, i.e. it is injective and $[x]\sim [y]$ if and only if $x\sim y,$ $\foa x,y\in Q_r.$ Then one cannot change the topology structure or weights inside $Q_{r-1}\subset T_A,$ such that the Ollivier curvatures are still non-negative. 
			
		\end{cor}
	
		The embedding condition ensures that we can open up the torus without changing the local structure of $Q_r\subset T_A,$ and the statement about changing the structure inside a closed cube is to mimic the connected sum in the discrete sense.

\section{Appendix}
%In this appendix, we consider graphs with vertex weights.	
%In particular, we give some counterexamples for the positive mass conjecture allowing non-constant vertex weights. 

In this appendix, we will explain why Theorem~\ref{thm:main1} and Theorem~\ref{thm:mainric} fail when allowing for general vertex weights, i.e., there is a function $m:V\rightarrow(0,\infty)$, and the Laplacian is generalized to
$$
\Delta f(x)=\frac{1}{m(x)}\sum_{y\sim x}w(x,y)(f(y)-f(x)).
$$
The Ollivier curvature of a graph with general vertex weight $m$ is defined as in $(\ref{limit-freeR})$, see \cite{limit-free} for details. In our setting above, we were implicitly assuming $m\equiv 1$. 

If we allow varying vertex weights, there might be multiple vertices sharing the same coordinates, as shown in the following example.
\begin{example}
	 %One can refer to \cite{limit-free} for detailed definitions related to varying vertex weights.
	Consider the graph as in  Figure \ref{fig:counter-example1}.
	\begin{figure}[htbp]
	\centering
	% drawn by GeoGebra
	\definecolor{qqqqff}{rgb}{0,0,1}
	\definecolor{ccqqqq}{rgb}{1,0,0}
	\definecolor{xdxdff}{rgb}{0.49019607843137253,0.49019607843137253,1}
	\definecolor{ududff}{rgb}{0.30196078431372547,0.30196078431372547,1}
	\begin{tikzpicture}[line cap=round,line join=round,>=triangle 45,x=1cm,y=1cm,scale=0.8]
		\clip(1,2) rectangle (9,7);
%		\selectcolormodel{gray}
		\draw [line width=1pt,domain=-0.2583784073871919:10.07484580550937] plot(\x,{(-0--3*\x)/2});
		\draw [line width=1pt,domain=-0.2583784073871919:10.07484580550937] plot(\x,{(--12-0*\x)/4});
		\draw [line width=1pt,domain=-0.2583784073871919:10.07484580550937] plot(\x,{(--24-0*\x)/4});
		\draw [line width=1pt,domain=-0.2583784073871919:10.07484580550937] plot(\x,{(--12-3*\x)/-2});
		\draw [line width=1pt,domain=-0.2583784073871919:3.005153434290766] plot(\x,{(-13.48784088881639-0.019493133348718494*\x)/-3.005153434290766});
		\draw [line width=1pt,domain=6.998976237781286:10.07484580550937] plot(\x,{(--17.896071425136057--0.0021910878766844277*\x)/3.9816713831586616});
		\draw [line width=1pt,domain=6:10.07484580550937] plot(\x,{(-6--3*\x)/2});
		\draw [line width=1pt,domain=-0.2583784073871919:4] plot(\x,{(--6-3*\x)/-2});
		\draw [line width=1pt,color=ccqqqq] (3.005153434290766,4.507730151436149)-- (5,5);
		\draw [line width=1pt,color=ccqqqq] (3.005153434290766,4.507730151436149)-- (5,4);
		\draw [line width=1pt,color=ccqqqq] (5,4)-- (4,3);
		\draw [line width=1pt,color=ccqqqq] (5,4)-- (6.998976237781286,4.49846435667193);
		\draw [line width=1pt,color=ccqqqq] (6.998976237781286,4.49846435667193)-- (5,5);
		\draw [line width=1pt,color=ccqqqq] (5,5)-- (6,6);
		\draw [line width=1pt,color=ccqqqq] (5,5)-- (4,3);
		\draw [line width=1pt,color=ccqqqq] (5,4)-- (6,6);
		\draw [line width=1pt,color=qqqqff] (5,5)-- (5,4);
		\begin{scriptsize}
			\draw [fill=ududff] (2,3) circle (2.5pt);
			\draw [fill=ududff] (4,6) circle (2.5pt);
			\draw [fill=ududff] (6,3) circle (2.5pt);
			\draw [fill=ududff] (8,6) circle (2.5pt);
			\draw [fill=xdxdff] (3.005153434290766,4.507730151436149) circle (2.5pt);
			%				\draw [fill=xdxdff] (0,4.488237018087431) circle (2.5pt);
			\draw [fill=xdxdff] (6.998976237781286,4.49846435667193) circle (2.5pt);
			\draw [fill=ududff] (10.980647620939948,4.500655444548614) circle (2.5pt);
			\draw [fill=xdxdff] (6,6) circle (2.5pt);
			\draw [fill=ududff] (8,9) circle (2.5pt);
			\draw [fill=xdxdff] (4,3) circle (2.5pt);
			\draw [fill=xdxdff] (2,0) circle (2.5pt);
			\draw [fill=ududff] (5,5) circle (2.5pt);
			\draw [fill=ududff] (5,4) circle (2.5pt);
			\draw[color=qqqqff] (4.9,5.3) node {$a$};
			\draw[color=qqqqff] (5.1,3.7) node {$b$};
			\end{scriptsize}
		\end{tikzpicture}
		\caption{The multi-vertices problem in $2$-dimension. The outside is the standard grid graph.} \label{fig:counter-example1}
	\end{figure}
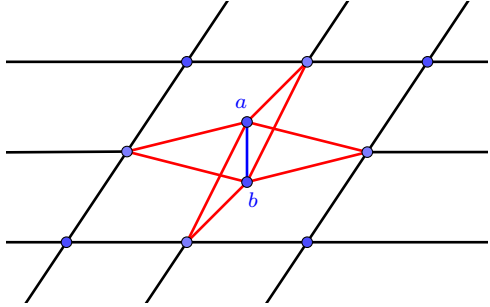
	
	Assume that all vertex weights are $1$ except for the vertices $a,b$ with weights of $\frac{1}{2}.$ Let weights of black edges be $1$. The weights are $\frac{1}{2}$ for the red edges, and $\frac{1}{4}$ for the blue edge $\{a,b\}$. By standard calculations similar to the Proposition \ref{prop:ric}, one can show that the curvatures are non-negative everywhere. Specifically, the Ricci curvatures are zero for all edges except the blue edge $\{a,b\}$, which has the curvature 5.
	
	Hence this graph is not a standard grid graph, though it has non-negative Ricci curvature and zero energy.
	
\end{example}

Moreover, the positive mass theorem fails when the dimension $n=1.$ The multi-vertices problem still happens, see the following example. This indicates that the constraints of curvature conditions on topology structure have to be carried out through multiple directions.

\begin{example}
	Consider the graph as shown in Figure \ref{fig:counter-example2} with vertex weights $m\equiv 1$.
	
	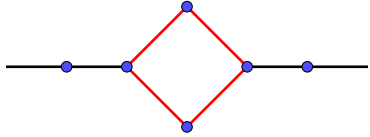
\begin{figure}[htbp]
		\centering
		% drawn by GeoGebra
		\definecolor{qqqqff}{rgb}{0,0,1}
		\definecolor{ccqqqq}{rgb}{1,0,0}
		\definecolor{xdxdff}{rgb}{0.49019607843137253,0.49019607843137253,1}
		\definecolor{ududff}{rgb}{0.30196078431372547,0.30196078431372547,1}
		\begin{tikzpicture}[line cap=round,line join=round,>=triangle 45,x=1cm,y=1cm,scale=0.8]
%			\selectcolormodel{gray}
			\clip(3,4) rectangle (9,8);
			\draw [line width=1pt,domain=7:9.937669330469639] plot(\x,{(--6-0*\x)/1});
			\draw [line width=1pt,domain=1.6980386794402798:5] plot(\x,{(-6-0*\x)/-1});
			\draw [line width=1pt,color=ccqqqq] (5,6)-- (6,7);
			\draw [line width=1pt,color=ccqqqq] (6,7)-- (7,6);
			\draw [line width=1pt,color=ccqqqq] (7,6)-- (6,5);
			\draw [line width=1pt,color=ccqqqq] (6,5)-- (5,6);
			\begin{scriptsize}
				\draw [fill=ududff] (7,6) circle (2.5pt);
				\draw [fill=ududff] (8,6) circle (2.5pt);
				\draw [fill=ududff] (5,6) circle (2.5pt);
				\draw [fill=ududff] (4,6) circle (2.5pt);
				\draw [fill=ududff] (6,7) circle (2.5pt);
				\draw [fill=ududff] (6,5) circle (2.5pt);
			\end{scriptsize}
		\end{tikzpicture}
		\caption{The multi-vertices problem in the 1-dimension case. The outside is a standard grid line.} \label{fig:counter-example2}
	\end{figure}

	Assume that all edge weights are $1$ except for the red edges with weights of $\frac{1}{2}$. It is a Ricci-flat graph with zero energy, but it is not the standard grid line $\Z.$
	
\end{example}

\textbf{Acknowledgements.} B. Hua is supported by NSFC, no.11831004, and by Shanghai Science and Technology Program [Project No. 22JC1400100].
	
\bibliographystyle{plain}
\bibliography{PMT}

\end{document}